\documentclass[10pt]{article}

\usepackage{amsmath, amssymb, mathrsfs, vmargin, theorem, dsfont, color, amsfonts, bbm, mathtools}
\usepackage[utf8]{inputenc}
\usepackage[T1]{fontenc}
\usepackage[english]{babel}
\usepackage{authblk}
\usepackage{endnotes}
\usepackage{graphicx}
\usepackage[margin=1.25in]{geometry}
\usepackage{cases}
\usepackage[numbers]{natbib}
\usepackage{float}
\usepackage{subcaption}

\setmarginsrb{30mm}{30mm}{25mm}{30mm}{0pt}{0pt}{0pt}{26pt}

\theoremheaderfont\scshape
\newtheorem{theorem}{Theorem}[section]
\newtheorem{proposition}{Proposition}[section]

\newtheorem{lemma}{Lemma}[section]
\theorembodyfont{\rmfamily}
\newtheorem{definition}{Definition}[section]
\newtheorem{ansatz}{Ansatz}[section]
\newtheorem{remark}{Remark}[section]

\newenvironment{proof}[1][]{\noindent\textit{Proof#1.} }{\vskip\baselineskip}

\renewcommand\hat[1]{\widehat{#1}}

\newcommand\argmax{\operatorname{argmax}}

\newcommand\qed{\hfill$\Box$}

\newcommand\bcdot{\ensuremath{%
  \mathchoice%
   {\mskip\thinmuskip\lower0.2ex\hbox{\scalebox{1.5}{$\cdot$}}\mskip\thinmuskip}}%
   {\mskip\thinmuskip\lower0.2ex\hbox{\scalebox{1.5}{$\cdot$}}\mskip\thinmuskip}%
   {\lower0.3ex\hbox{\scalebox{1.2}{$\cdot$}}}%
   {\lower0.3ex\hbox{\scalebox{1.2}{$\cdot$}}}%
}

\numberwithin{equation}{section}

\DeclarePairedDelimiter\abs{\lvert}{\rvert}
\DeclarePairedDelimiter\norm{\lVert}{\rVert}

\begin{document}
\title{Nonzero-sum stochastic differential games  between\\ an impulse controller and a stopper}
\author{Luciano Campi} 
\author{Davide De Santis}
\affil{Department of Statistics \\ London School of Economics and Political Science}
\date{\today}
\maketitle

\begin{abstract}
We study a two-player nonzero-sum stochastic differential game where one player controls the state variable via additive impulses while the other player can stop the game at any time. The main goal of this work is characterize Nash equilibria through a verification theorem, which identifies a new system of quasi-variational inequalities whose solution gives equilibrium payoffs with the correspondent strategies. Moreover, we apply the verification theorem to a game with a one-dimensional state variable, evolving as a scaled Brownian motion, and with linear payoff and costs for both players. Two types of Nash equilibrium are fully characterized, i.e. semi-explicit expressions for the equilibrium strategies and associated payoffs are provided. Both equilibria are of threshold type: in one equilibrium players' intervention are not simultaneous, while in the other one the first player induces her competitor to stop the game. Finally, we provide some numerical results describing the qualitative properties of both types of equilibrium.\\

\noindent
\textbf{Keywords and phrases:} controller-stopper games, stochastic differential games, impulse controls, quasi-variational inequalities, Nash equilibrium. \medskip

\noindent\textbf{MSC 2010 classification:} 91A15, 91B70, 93E20.
\end{abstract}

\section{Introduction}
Controller-stopper games are two-player stochastic dynamic games whose payoffs depend on the evolution over time of some state variable, one player can control its dynamics while the other player can stop the game. The study of these games  started with Maitra and Sudderth's work \cite{maitra1996gambler} on a zero-sum discrete time setting. Later on, many authors investigated such games in continuous time, especially in the zero-sum case, while very little has been done in the nonzero-sum. Indeed, apart from Karatzas and Sudderth \cite{karatzas2006stochastic}, 
all the other articles focus on the zero-sum case and in all of them the controller uses regular controls, i.e. absolutely continuous for the Lebesgue measure. Here we mention Karatzas and Sudderth \cite{karatzas2001controller}, who derived the explicit solution for a game with a one-dimensional diffusion with absorption at the endpoints of a bounded interval as a state process; Karatzas and Zamfirescu \cite{karatzas2006martingale, karatzas2008martingale} developed a martingale approach to a general class of controller-stopper games, while Bayraktar and Huang \cite{bayraktar2013multidimensional} showed that the value functions of such games is the unique viscosity solution to an appropriate Hamilton-Jacobi-Bellman equation. Moreover, Hernandez et al. \cite{zervos2015} have analysed the case when the controller plays singular controls and derived a set of variational inequalities characterizing the games value functions.  
On the whole, this class of games is motivated by a variety of applications in finance, insurance and economics. In view of this we quote Bayraktar et al. \cite{BayraktarE.2010Osfd} on convex risk measures, Nutz and Zhang \cite{nutz2015optimal} on sub-hedging of American options under volatility uncertainty, Bayraktar and Young \cite{BayraktarErhan2011Prot} on minimization of lifetime ruin probability and Karatzas and Wang \cite{KaratzasI.2000ABOo} on pricing and hedging of American contingent claims among others.

Here, we consider the case of a controller facing fixed and proportional costs every time he moves the state variable, so that intervening continuously over time is clearly not feasible for him. In this context, the controller will make use of impulse controls, which are sequences of interventions times and corresponding intervention sizes, describing when and by how much will the controlled process be shifted. This kind of controls look like the natural choice in many concrete applications, from finance to energy markets and to real options. For this reason, they have been experiencing a comeback due to a demand for more realistic financial models (e.g. fixed transaction costs and liquidity risk), see for instance \cite{ Basei18, bruder2009impulse,cadenillas1999optimal, chen2013, chevalier2016optimal, vath2007model}. 

Impulse controls have been studied in stochastic differential games as well and, as in the controller-stopper case, most of the research has been done in the zero-sum framework. For this reason, it is worth mentioning the work by A\"id et al. \cite{aid2016nonzero}, who developed a general model for non-zero sum impulse games implementing a verification theorem which provides an appropriate system of quasi-variational inequalities for the equilibrium payoffs and related strategies of the two players. Thereafter, Ferrari and Koch \cite{ferrari2017strategic} produced a model of pollution control where the two players, the regulator and the energy producer, are assumed to face proportional and fixed costs and, as such, play an impulse nonzero-sum game which admits an equilibrium under some suitable conditions. Lastly, Basei et al. \cite{Basei19} studied the mean field game version of the nonzero-sum impulse game in \cite{aid2016nonzero} and proved the existence of $\epsilon$-Nash equilibrium for the corresponding $N$-player game. Regarding the zero-sum case, here we quote Cosso \cite{cosso2013stochastic}, who examined a finite time horizon two-player game where both players act via impulse control strategies and showed that such games have a value which is the unique viscosity solution of the double-obstacle quasi-variational inequality. Furthermore, Azimzadeh \cite{azimzadeh2017zero} considered an asymmetric setting with one participant playing a regular control while the opponent is playing an impulse control with precommitment, meaning that at the beginning of the game the maximum number of impulses is declared, and proved that such a game has a value in the viscosity sense.

This paper is at the crossroad of the two streams of research we have discussed above: stopper-controller games and impulse games. Indeed, we study an impulse controller-stopper nonzero-sum game, focusing on the mathematical properties of Nash equilibria, while application to economics and finance are postponed to future research.
Turning to the game's description, we consider a nonzero-sum stochastic differential game between two players, P1 and P2, where P1 can use impulse controls to affect a continuous-time stochastic process $X$ while P2 can stop the game at any time. When P1 does not intervene, we assume $X$ to diffuse according to a time homogeneous multidimensional diffusion process. Moreover, every time P1 intervenes she faces a cost, say $\phi(X_{\tau_n -}, \delta_n)$ which may depend on the impulse size and on the state variable right before intervention. At the same time P2 can get something, say $\psi(X_{\tau_n -}, \delta_n)$. P1's strategy is any sequence $u=(\tau_n,\delta_n)_{n \ge 1}$, of intervention times, $\tau_n$, and corresponding impulses $\delta_n$, whereas P2's strategy is any $[0,\infty]$-valued stopping time $\eta$. Both players want to maximize their expected payoffs which are defined for every initial state $x \in \mathbb{R}^d$ and every couple $(u, \, \eta)$ as follows: 
\begin{align*}
J_1(x; u, \eta) &= \mathbb E \left[ \int_0 ^\eta e^{-r_1  t} f (X_t ^{x,u}) dt - \sum_{n: \tau_n \leq \eta}e^{-r_1 \tau_n}\phi(X_{\tau_n -}, \delta_n) + e^{-r_1 \eta} h(X_{\eta})\mathbf 1_{(\eta < \infty)}  \right]  ,\\
J_2(x; u, \eta) &= \mathbb E \left[ \int_0 ^\eta e^{-r_2 t} g (X_t ^{x,u}) dt + \sum_{n: \tau_n \leq \eta}e^{-r_2 \tau_n} \psi(X_{\tau_n -}, \delta_n) + e^{-r_2 \eta} k(X_{\eta})\mathbf 1_{(\eta < \infty)}  \right] ,
\end{align*}
so apart from the costs (and gains) induced by the intervention of P1, the players have running as well as terminal gains.
We are going to adopt a PDE-based approach to characterize the Nash equilibria of this game, identifying a suitable system of quasi-variational inequalities (QVI's, for short) whose solution will give equilibrium payoffs. In this setting, the relevant QVI's system is as follows
\begin{align*} 
&\mathcal{M}V_1 - V_1 \leq 0 & \mbox{ in } \mathbb{R}^d \\
&V_2 - k \geq 0& \mbox{ in } \mathbb{R}^d \\
&\mathcal{H}V_2 - V_2 = 0& \mbox{ in } \{  \mathcal{M}V_1 - V_1 = 0 \} \\
&V_1 = h & \mbox{ in } \{V_2 =k \} \\
&\max\{ \mathcal{A}V_1 - r_1V_1 + f, \mathcal{M}V_1 - V_1 \} = 0 & \mbox{ in } \{ V_2 > k \} \\ 
&\max\{ \mathcal{A}V_2 - r_2V_2 + g, k -  V_2 \} = 0 & \mbox{ in } \{ \mathcal{M}V_1 - V_1  < 0 \}  
\end{align*}
where $\mathcal{M}$ and $\mathcal{H}$ are suitable intervention operators defined in Section \ref{qvisec}. One of the main contributions of this paper consists in the Verification Theorem \ref{oksendal-sulem verification} establishing that if two functions $V_1$ and $V_2$ are regular enough and they are solution to the system above, 
then they coincide with some equilibrium payoff functions of the game and a characterization of the related equilibrium strategies is possible. \medskip

Furthermore, building on the verification theorem, we present an example of solvable impulse controller and stopper game. More in detail, we consider a game with a one-dimensional state variable $X$, modelled as a real-valued (scaled) Brownian motion. Both players have linear running payoffs. When P1 intervenes, he faces a penalty while P2 faces a gain, both characterized by a fixed and a variable part, proportional to the size of the impulse. Moreover, when P2 stops the game, he may suffer a loss proportional to the state variable, while P1 might gain something proportional to $X$ as well. Hence, the players objective functions are:
\begin{align*}
J_1(x; u, \eta) &= \mathbb{E}_x \left[ \int_0^\eta e^{-r t} (X_t - s) dt - \sum_{n: \tau_n \leq \eta} e^{-r \tau_n} (c + \lambda \abs{\delta_n}) + a e^{- r \eta}  X_{\eta}\mathbbm{1}_{\{\eta < \infty\}}\right], \\
J_2(x; u, \eta) &= \mathbb{E}_x \left[ \int_0^\eta e^{-r t} (q - X_t ) dt + \sum_{n: \tau_n \leq \eta} e^{-r \tau_n} (d + \gamma \abs{\delta_n}) - b e^{- r \eta}X_{\eta}\mathbbm{1}_{\{\eta < \infty\}}\right] ,
\end{align*}
with suitable coefficients satisfying some assumptions that will be specified later.
Some preliminary heuristics on the QVIs above leads us to consider two pairs of candidates for the functions $V_i$. Then, a careful application of the verification theorem shows that such candidates actually coincide with some equilibrium payoff functions. In particular, we are able to identify two kinds of Nash equilibria, both of threshold type, that can be shortly described as follows:
\begin{enumerate}
\item in the first type of equilibrium, P1 intervenes when the state $X$ is smaller than some threshold $\bar x_1$ and moves the process to some endogenously determined target $x_1^*$, while P2 terminates the game when the state $X$ is bigger than some $\bar x_2$; in this kind of equilibrium the optimal target of P1, $x_1 ^*$, is strictly smaller than $\bar x_2$, so the two players intervene separately.
\item In the second type, P1 intervenes when the state $X$ is smaller than some (possibly different) threshold $\bar x_1$ and move the state variable to the intervention region of P2, who is then forced by P1 to end the game. In this case, players' interventions are simultaneous.
\end{enumerate}
We provide quasi-explicit expressions for the value functions and for the thresholds $\bar x_i$, $x_1^*$ for both equilibria. Finally, we perform some numerical experiments providing several cases when one of the two equilibria emerges. The question if there are cases when the two types of equilibria can coexist is still open. 

The paper is organised as follows. Section \ref{gamesetting} gives the general formulation of impulse controller and stopper game, in particular the notion of admissible strategies, and more importantly we state and prove a verification theorem giving sufficient condition in terms of the system of QVIs for a given couple of payoffs to be a Nash equilibrium. In Section \ref{game} we consider the one-dimensional example with linear payoffs and provide quasi-explicitly characterizations for the two types of Nash equilibria sketched above. Finally, some numerical experiments illustrate the qualitative behaviour of such equilibria. 

\section{Description of the game}\label{gamesetting}

In this section we have gathered all main theoretical results on a general class of nonzero-sum impulse controller and stopper games. We start with a detailed description of the game, together with all technical assumptions and the definition of admissible strategies.

Let $( \Omega, \mathbb{F}, \mathbb{P})$ be a probability space equipped with a complete and right-continuous filtration $\mathbb{F} = (\mathcal{F}_t)_{t \geq 0}$. 
On this space we consider the uncontrolled state variable $X\equiv X^x$ defined as solution of the following time-homogeneous SDE:
\begin{equation}\label{dynamic}
 dX_t = b(X_t)dt + \sigma(X_t)dW_t , \quad X_{0^-} =x,
 \end{equation}
where $(W_t)_{t\geq0}$ is an $\mathbb{F}$-Brownian motion and the coefficients $b: \mathbb{R}^d \rightarrow \mathbb{R}^d$ and $\sigma: \mathbb{R}^d \rightarrow \mathbb{R}^{d\times k}$ are assumed to be globally Lipschitz continuous, i.e. there exists a constant $C>0$ such that for all $x_1, \, x_2 \in \mathbb{R}^d$ we have:
$$
\abs{b(x_1) - b(x_2)} + \abs{\sigma(x_1) - \sigma(x_2)} \leq C\abs{x_1 - x_2},
$$
so that existence of a unique strong solution is granted and $X$ is well-defined.
We consider two players, that we call P1 and P2. Equation \eqref{dynamic} describes the evolution of the state process in case of no intervention from both players. During the game, P1 can affect $X$'s dynamics applying some impulse $\delta \in Z$ in an additive fashion, moving the state variable from $X_{\tau -}$ to its new value $X_{\tau} = X_{\tau^-} +  \delta$, where $\tau$ denotes the intervention time. The controlled state variable will be denoted by $X^{x,u}$:
$$
X_t^{x, u} = x + \int_0^t b(X_s^{x, u}) ds + \int_0^t \sigma(X_s^{x, u}) dW_s + \sum_{n: \tau_n \leq t} \delta_n , \quad t \ge 0.
$$
On the other hand, P2 can stop the game by choosing any stopping time $\eta$ with values in $[0,\infty]$. We, now, give a proper definition of such strategies.
\begin{definition} Let $Z \subseteq \mathbb{R}^d$ be given. P1's strategy is any sequence $u=(\tau_n, \delta_n)_{n \geq 0}$, where $(\tau_n)_{n \geq 0}$ is a sequence of stopping times such that $0 = \tau_0 < \tau_1 < \tau_2<\ldots<\tau_n \uparrow \infty$ and $\delta_n \in L^0(\mathcal{F}_{\tau_n})$ with $\delta_n: \mathbb{R}^d \rightarrow Z$. P2's strategy is any stopping time $\eta \in \mathcal{T}$, where $\mathcal{T}$ is the set of all $[0,\infty]$-valued $\mathbb{F}$-stopping times. 
\end{definition}

\begin{remark}
We observe that simultaneous interventions are possible in this game. This is in contrast with games where both players intervenes with impulses, where simultaneous interventions are usually not allowed since they would be very difficult to handle with from a modelling perspective (cf. \cite{aid2016nonzero}). On the other hand here, due to the different nature of the strategies for the two players, one can safely allow for simultaneous actions. This has an interesting consequence on our analysis, as we will see in the linear game of the next section that at least two types of Nash equilibria are possible and in one of them P1 induces P2 to stop instantaneously.  
\end{remark}

The players want to maximize their respective objectives, featuring each of them three discounted terms: a running payoff, P1's intervention cost/gain and a terminal payoff. The players' discount factors can be different of each other. More precisely, for each $i=1,2$, $r_i > 0$ denotes the discount rate of player $i$, $f,g: \mathbb{R}^d \rightarrow \mathbb{R}$ are their running payoffs, $h,k: \mathbb{R}^d \rightarrow \mathbb{R}$ their terminal payoffs and $\phi, \psi:\mathbb{R}^d \times Z \rightarrow \mathbb{R}$ are the intervention cost and gain, respectively. Throughout the whole paper, we will work under the assumption that all these functions are continuous.
Hence, we can define the payoffs as follows.
\begin{definition}
Let $x \in \mathbb{R}^d$, let $(u, \eta)$ be a pair of strategies. Provided that the right-hand sides exist and are finite we set: 
\begin{align*} 
J_1(x; u, \eta) &=\mathbb{E}_x \left[ \int_0 ^\eta e^{-r_1 t} f (X_t ^{x, u}) dt - \sum_{n: \tau_n \leq \eta}e^{-r_1\tau_n}\phi(X^{x, u}_{\tau_n^-}, \delta_n) + e^{-r_1\eta} h(X^{x, u}_{\eta})\mathbbm{1}_{(\eta < \infty)}  \right]  \\
J_2(x; u, \eta) &= \mathbb{E}_x \left[ \int_0 ^\eta e^{-r_2 t} g (X_t ^{x, u}) dt + \sum_{n: \tau_n \leq \eta}e^{-r_2 \tau_n} \psi(X^{x, u}_{\tau_n^-}, \delta_n) + e^{-r_2 \eta} k(X_{\eta}^{x, u})\mathbbm{1}_{(\eta < \infty)}  \right] ,
\end{align*}
\end{definition}
where the subscript in the expectation denotes the conditioning with respect to the starting point.

In order for $J_1$ and $J_2$ to be well defined, we now introduce the set of admissible strategies.
\begin{definition}\label{admissibility} Let $x \in \mathbb{R}^d$ be some initial state and let $(u ,\eta)$ be some strategy profile. We say that the pair $(u, \eta)$ is $x$-admissible if:
\begin{enumerate}
\item the following random variables are all in $L^1(\Omega)$:
\begin{eqnarray*}
\int_0^\infty e^{- r_1 t} \abs{f(X_t^{x, u})} dt, && \int_0^\infty e^{- r_2 t} \abs{g(X_t^{x, u})} dt,\\
e^{- r_1\eta} \abs{h(X_{\eta}^{x, u})}, && e^{- r_2\eta} \abs{k(X_{\eta}^{x, u})}, \\
\sum_{k: \tau_k  \leq \infty} e^{- r_1\tau_k} \abs{\phi(X_{\tau_k^-}^{x,u}, \delta_k)}, && \sum_{k: \tau_k  \leq \infty} e^{- r_2\tau_k} \abs{\psi(X_{\tau_k^-}^{x, u}, \delta_k)};
\end{eqnarray*}
\item for each $p \in \mathbb{N}$, the random variable $\norm{X^{x, u}}_\infty = \sup_{t \geq 0} \abs{X_t^{x, u}}$ is in $L^p(\Omega)$.
\end{enumerate}
We denote by $\mathcal{A}_x$ the set of all $x$-admissible pairs.
\end{definition}
\begin{remark} Notice that, as it is formulated above, admissibility is a joint condition on the strategies of both players. Under condition (ii) above and if all functions $f$, $g$, $h$, $k$, $\phi$ and $\psi$ have at most polynomial growth in their respective variables, the set of all jointly admissible strategies can be expressed as $\mathcal{A}_x^1 \times \mathcal{A}_x^2 = \mathcal{A}_x$, where $\mathcal{A}_x^i$ denotes Pi's set of (individually) admissible strategies for $i=1,2$, and is defined as follows: $\mathcal A_x ^1$ is the set of all P1's strategies $u=(\tau_n, \delta_n)_{n \geq 0}$ such that $\sum_{n \ge 0} |\delta_n| \in L^p (\Omega)$ for all $p \ge 1$, while $\mathcal A_x ^2$ is the set of all $[0,\infty]$-values stopping times.

Indeed, for P1's strategies for instance, using classical a-priori $L^p$-estimates of the (uncontrolled) state variable, there exists a constant $c>0$ such that
\begin{align*}
\mathbb{E}\left[ e^{- r_1 \eta} \abs{h(X_{\eta})}\right] \leq  c \mathbb{E} \left[ e^{- r_1 \eta} (1 + \abs{X_{\eta}}^p)\right] 
\leq c(1+\mathbb{E}[ \norm{X}^p _\infty ] )< \infty .
\end{align*}
Moreover, similar estimates can be performed for the other expectations in Definition \ref{admissibility}(i).
\end{remark}
 
We conclude this section with the classical definition of Nash equilibrium and the corresponding equilibrium payoffs.
\begin{definition}{(Nash Equilibrium)} Given $x \in \mathbb{R}^d$, we say that $(u, \eta) \in \mathcal{A}_x$ is a Nash equilibrium if 
\begin{align*} 
J_1(x; u^*, \eta^*) &\geq J_1(x; u, \eta^*) ,\quad \text{for all $u$ s.t.} \; (u, \eta^*) \in \mathcal{A}_x ,\\ 
J_2(y; u^*, \eta^*) &\geq J_2(x; u^*, \eta) , \quad \text{for all $\eta$ s.t.} \; (u^*, \eta) \in \mathcal{A}_x
.\end{align*}
Finally, the equilibrium payoffs of any Nash equilibrium  $(u^*, \eta^*) \in \mathcal{A}_x$ are defined as  
$$
V_i(x):=J_i(x; u^*, \eta^*), \quad i=1,2.
$$
\end{definition}

\subsection{The system of quasi-variational inequalities} \label{qvisec}

Now, we are going to introduce the differential problem that will be satisfied by the equilibrium payoff functions of our game.
Let $V_1, V_2 : \mathbb R^d \to \mathbb R$ be two measurable functions such that
\begin{equation} \label{deltamax}
\{ \delta(x) \} = \argmax_{\delta \in Z} \{ V_1(x + \delta) - \phi(x, \delta)\}, \quad x \in \mathbb{R}^d,
\end{equation}
for some measurable function $\delta : \mathbb{R}^d \to Z$. Moreover we define the following two intervention operators:
\begin{align}
\mathcal{M}V_1(x) &= V_1(x+\delta(x)) - \phi(x, \delta(x))\label{intervention cost operator} , \\
\mathcal{H}V_2(x) &= V_2(x+ \delta(x)) + \psi(x, \delta(x))\label{intervention gain operator} ,
\end{align}
for each $x \in \mathbb{R}^d$. 

The expressions in \eqref{deltamax}, \eqref{intervention cost operator} and \eqref{intervention gain operator} have the following natural interpretation:
\begin{itemize}
\item[\eqref{deltamax}] let $x$ be the current state of the process, if P1 intervenes immediately with impulse $\delta(x)$, P1's payoff after intervention changes to $V_1(x+\delta(x)) - \phi(x, \delta(x))$, given by the payoff in the new state minus the intervention cost. Therefore, $\delta(x)$ in \eqref{deltamax} is the optimal impulse that P1 would apply in case of intervention. 
\item[\eqref{intervention cost operator}] $\mathcal{M}V_1(x)$ represents P1's payoff just after her intervention.
\item[\eqref{intervention gain operator}] similarly, $\mathcal{H}V_2(x)$ represents P2's payoff following P1's intervention.
\end{itemize}
Moreover, for any $V \in C^2(\mathbb{R}^d)$ we can consider the infinitesimal generator of the uncontrolled state variable $X$:
$$
\mathcal{A}V =b \cdot \nabla V + \frac{1}{2} \text{tr}(\sigma \sigma^t D^2 V),
$$
where $b, \sigma$ are as in \eqref{dynamic}, $\sigma^t$ denotes the transposed of $\sigma$, $\nabla V$ and $D^2 V$ are the gradient and the Hessian matrix of $V$, respectively. We are interested in the following quasi-variational inequalities (QVI's, for short) for $V_1, V_2$:
\begin{align}
&\mathcal{M}V_1 - V_1 \leq 0 & \mbox{ everywhere} \label{P1 intervention}\\
&V_2 - k \geq 0& \mbox{ everywhere} \label{P2 intervention} \\
&\mathcal{H}V_2 - V_2 = 0& \mbox{ in } \{  \mathcal{M}V_1 - V_1 = 0 \} \label{P2 no deviation}\\
&V_1 = h & \mbox{ in } \{V_2 =k \} \label{boundary}\\
&\max\{ \mathcal{A}V_1 - r_1V_1 + f, \mathcal{M}V_1 - V_1 \} = 0 & \mbox{ in } \{ V_2 > k \} \label{P1 control}\\ 
&\max\{ \mathcal{A}V_2 - r_2V_2 + g, k -  V_2 \} = 0 & \mbox{ in } \{ \mathcal{M}V_1 - V_1  < 0 \}  \label{P2 control}
\end{align}
Each part of the QVI's system above can be interpreted in the following way: 
\begin{itemize}
\item[\eqref{P1 intervention}] it means that is not always optimal for P1 to intervene and it is a standard condition in impulse control theory \cite{cadenillas1999optimal, bertola2016classical};
\item[\eqref{P2 intervention}] if the current state is $x$ and P2 chooses to stop the game, i.e. $\eta = 0$, he gains $k(x)$ and since this is a suboptimal strategy, we have $V_2(x) \geq k(x)$ for all $x \in \mathbb{R}^d$;
\item[\eqref{P2 no deviation}] by definition of Nash equilibrium we expect that P2 does not lose anything when P1 intervenes as in \cite{aid2016nonzero}, otherwise P2 would like to deviate, by contradicting the notion of equilibrium;
\item[\eqref{P1 control}] before P2 stops the game, P1 plays as in a classic impulse control problem (e.g. \cite{cadenillas1999optimal});
\item[\eqref{P2 control}] similarly as above, when P1 does not intervene P2 solves his own optimal stopping problem (e.g. \cite{chen2013nonzero}).
\end{itemize}
After all this preparation, we are ready to move to our main result, which is a verification theorem linking Nash equilibria and solutions to the QVI system above.

\subsection{The verification theorem}

In this subsection, we state and proof our main verification theorem. This result will be key in order to compute Nash equilibria in specific examples.

\begin{theorem}\label{oksendal-sulem verification} 
Let $V_1, V_2 : \mathbb{R}^d \to \mathbb{R}$ be two given functions. Assume that \eqref{deltamax} holds and set 
\[ \mathcal{C}_1:=\{\mathcal{M}V_1 - V_1 <0\}, \quad \mathcal{C}_2:=\{V_2 - k>0\},\] 
with $\mathcal{M}V_1$ as in \eqref{intervention cost operator}. Moreover, assume that:
\begin{itemize}
\item $V_1$ and $V_2$ are solutions of the system of QVIs;
\item $V_i \in C^2(\mathcal{C}_j \setminus \partial \mathcal{C}_i) \cap C^1(\mathcal{C}_j) \cap C(\mathbb{R}^d)$,  
for $i \neq j$, and both functions have at most polynomial growth;
\item $\partial \mathcal{C}_i$ is a Lipschitz surface\footnote{I.e. it is locally the graph of a Lipschitz function.}, and $V_i$'s second order derivatives are locally bounded near $\partial \mathcal{C}_i$ for $i=1,2$.
\end{itemize}
Finally, let $x \in \mathbb{R}^d$ and assume that $(u^*, \eta^*) \in \mathcal{A}_x$, where $u^*=(\tau_n, \delta_n)_{n \geq 1}$ is given by
\[\tau_n = \inf \{ t > \tau_{n-1}: X_t \in  \mathcal{C}_1^c \},\quad \{\delta_n\}  = \argmax_{\delta \in Z} \{ V_1(X_{\tau_n ^-} + \delta) - \phi(X_{\tau_n ^-}, \delta)\}, \quad n \ge 0,\]
and
\[ \eta^* =  \inf \{ t \geq 0: V_2(X_t)=k(X_t)\},\]
with the convention $\tau_0 =0$.
Then, $(u^*, \eta^*)$ is a Nash Equilibrium and $V_i = J_i(x; u^*, \eta^*)$ for $i =1, 2$.
\end{theorem}

\begin{remark} First, we stress that, unlike usual impulse control problems, the candidates $V_1,V_2$ are not required to be twice differentiable everywhere, but only in $ \{ V_2 > k \}$ and $\{ \mathcal{M}V_1 - V_1  < 0 \}$ respectively. Moreover, we observe that for the equilibrium strategies in the theorem above the right-continuity of $(X_t^{x; u})_{t \geq 0}$ implies the following:
\begin{align}
&(\mathcal{M}V_1 - V_1)(X_s^{x, u^*})<0, &  \label{optimal intervention operator}\\
&\delta_k = \delta(X_{\tau_k^-}^{x, u^*}),& (\mbox{where } \delta(\cdot) \mbox{ is as in \eqref{deltamax}}) \label{optimal impulse}\\
&(\mathcal{M}V_1 - V_1)(X_{\tau_k^-}^{x, u^*})=0,& (\mathcal{H}V_2 - V_2)(X_{\tau_k^-}^{x, u^*})=0 \label{optimal intervention time}\\
&(V_2-k)(X_{\eta^*}) = 0, & (\mbox{on } \{\eta^* < \infty\}) \label{optimal stopping time}\\
&(V_2-k)(X_s)>0, & \mbox{ (when P2 plays $\eta^*$) }  \label{optimal stopping operator}
\end{align}
for every strategies $u$ and $\eta$ such that both $(u^*, \eta)$, $(u, \eta^*)$ belong to $\mathcal{A}_x$, for every $s \in [0, \eta)$ and every $\tau_k < \infty$.
\end{remark}

\begin{proof}
Let $V_i(x) = J_i(x; u^*, \eta^*)$ for $i=1,2$. By definition of Nash Equilibrium we have to prove that $V_1(x) \geq J_1(x; u, \eta^*)$ and $V_2(x) \geq J_2(x; u^*, \eta)$ for every $(u, \eta)$ such that both $(u^*, \eta)$, $(u, \eta^*)$ belong to $\mathcal{A}_x$. The proof is performed in three steps.\medskip

\noindent\emph{Step 1}: We show that $V_1(x) \geq J_1(x; u, \eta^*)$. Let $u$ be a strategy such that $(u, \eta^*) \in \mathcal{A}_x$. Thanks to the regularity assumptions and by approximation arguments of Theorem 2.1 in \cite{oksendal2005applied} (for more details see the proof of Theorem 3.3 in \cite{aid2016nonzero}), we can assume without loss of generality that $V_1 \in C^2(\mathcal{C}_2) \cap C(\mathbb{R}^d)$. 
For each $r > 0$ and $n \in \mathbb{N}$, we set $$
\tau_{r,n} = \tau_r \wedge n \wedge \eta^*
$$
with $\tau_r = \inf \{s >0 : X_s \not\in B(x,r)\}$, where $B(x,r)$ is an open ball with radius $r$ and centre in $x$. As usual, we adopt the convention $ \inf \emptyset = + \infty$. Applying It\^o's formula to $e^{-r_1s}V_1(X_s)$ between time zero and $\tau_{r,n}$ and taking conditional expectations on both sides give
$$
V_1(x) = \!\mathbb{E}_x \!\left[ e^{- r_1 \tau_{r, n}} V_1(X_{\tau_{r,n}}) - \int_0^{\tau_{r,n}} e^{- r_1 s} (\mathcal{A}V_1 - r_1V_1)(X_s)ds - \!\!\sum_{k: \tau_k \leq \tau_{r,n}}\!\! e^{- r_1 \tau_k}( V_1(X_{\tau_k})- V_1(X_{\tau_k^-}))\right].
$$ 
From \eqref{P2 control} it follows that
$$ (\mathcal{A}V_1 - r_1 V_1)(X_s) \leq - f(X_s)$$
for all $s \in [0, \eta^*)$. Moreover, using \eqref{P1 intervention} we also have:
\begin{align*}
V_1(X_{\tau_k^-}) \geq \mathcal{M}V_1(X_{\tau_k^-}) \geq V_1(X_{\tau_k^-}+ \delta) - \phi(X_{\tau_k^-}, \delta) = V_1(X_{\tau_k})- \phi(X_{\tau_k^-}, \delta).
\end{align*}
Therefore, 
$$V_1(x) \geq \mathbb{E}_x\left[ e^{-r_1 \tau_{r, n}} V_1(X_{\tau_{r,n}}) + \int_0^{\tau_{r,n}} e^{-r_1s}f(X_s)ds - \sum_{k: \tau_k \leq \tau_{r,n}} e^{- r_1 \tau_k} \phi(X_{\tau_k^-}, \delta_k)\right].$$
Observe that by admissibility we have
$$V_1(X_{\tau_{r,n}}) \leq C\left(1+ \abs{X_{\tau_{r,n}}}^p\right) \leq C\left(1+\norm{X}^p_\infty \right) \in L^1(\Omega),$$
for some constants $C>0$ and $p \in \mathbb{N}$. Thus, we can use dominated convergence theorem and pass to the limit, first as $r \rightarrow \infty$ and then for $n \rightarrow \infty$. Finally, because of \eqref{boundary}, we obtain
$$V_1(x) \geq \mathbb{E}_x\left[  \int_0^{\eta^*} e^{-r_1 s}f(X_s)ds - \sum_{k: \tau_k \leq \eta^*} e^{- r_1 \tau_k} \phi(X_{\tau_k^-}, \delta_k) + e^{- r_1 \eta^*} h(X_{\eta^*}) \mathbbm{1}_{\{\eta^* <\infty\}} \right] = J_1(x; u, \eta^*).$$

\noindent\emph{Step 2}: We show that $V_2(x) \geq J_2(x; u^*, \eta)$.  Let $\eta$ be a $[0,\infty]$-valued stopping time such that $(u^*, \eta) \in \mathcal{A}_x$. Thanks to regularity assumptions and by the same approximation argument as before, we can assume again without loss of generality that $V_2 \in C^2(\mathcal{C}_1) \cap C(\mathbb{R}^d)$. Arguing exactly as in Step 1 we obtain
$$
V_2(x) \!=\! \mathbb{E}_x \!\left[ e^{- r_2 \tau_{r, n}} V_2(X_{\tau_{r,n}})\! - \!\int_0^{\tau_{r,n}} e^{- r_2 s} (\mathcal{A}V_2\! - \!r_2V_2)(X_s)ds -\!\! \sum_{k: \tau_k \leq \tau_{r,n}} \!\!e^{- r_2 \tau_k}\left( V_2(X_{\tau_k})\!-\! V_2(X_{\tau_k^-})\right)\right],
$$ 
for the localizing sequence $\tau_{r,n} = \tau_r \wedge n \wedge \eta$ ($r > 0$, $n \in \mathbb{N}$), where $\tau_r = \inf \{s >0 : X_s \not\in B(x,r)\}$.
From \eqref{P1 control} it follows that
$$ (\mathcal{A}V_2 - r_2 V_2)(X_s) \leq - g(X_s)$$
for all $s \in [0, \eta)$. Moreover, due to \eqref{P2 no deviation} and \eqref{optimal intervention time} we obtain
\begin{align*}
V_2(X_{\tau_k^-}) = \mathcal{H}V_2(X_{\tau_k^-}) =V_2((X_{\tau_k^-} + \delta_k) + \psi((X_{\tau_k^-}, \delta_k) =V_2(X_{\tau_k})+\psi(X_{\tau_k^-}, \delta_k).
\end{align*}
Then,
$$V_2(x) \geq \mathbb{E}_x\left[ e^{-r_2 \tau_{r, n}} V_2(X_{\tau_{r,n}}) + \int_0^{\tau_{r,n}} e^{-r_2s}g(X_s)ds + \sum_{k: \tau_k \leq \tau_{r,n}} e^{- r_2 \tau_k} \psi(X_{\tau_k^-}, \delta_k)\right]$$
and as before we can use dominated convergence theorem and pass to the limit so that using \eqref{boundary} we obtain 
 $$V_2(x) \geq \mathbb{E}_x\left[  \int_0^{\eta} e^{-r_1 s}g(X_s)ds + \sum_{k: \tau_k \leq \eta} e^{- r_1 \tau_k} \psi(X_{\tau_k^-}, \delta_k) + e^{- r_2 \eta} k(X_{{\eta}}) \mathbbm{1}_{\{\eta <\infty\}} \right] = J_2(x; u^*, \eta).$$

\noindent\emph{Step 3}: Let $V_1(x) = J_1(x; u^*, \eta^*)$. We argue as in Step 1, with equalities instead of inequalities by the property of $u^*$. Similarly for P2 with $V_2(x) = J_2(x; u^*, \eta^*)$.
\qed
\end{proof}

\section{An impulse controller-stopper game with linear payoffs}  \label{game}

In the next sections \ref{example setting}-\ref{numerics} we provide an application of the Verification Theorem \ref{oksendal-sulem verification} to an impulse game with a one-dimensional state variable evolving essentially as a Brownian motion, which can be shifted by P1's impulses and stopped by P2, and where both players want to maximise linear payoffs. We find two types of Nash equilibria for this game, depending on whether P1 finds it convenient or not to force P2 to stop the game. For both types, we provide quasi-explicit expressions for the equilibrium payoff functions and related strategies. Our findings will be illustrated by some numerical examples.
 
\subsection{Setting} \label{example setting}

We are in a more specific setting than before. This time, the state  variable is one-dimensional, while the players have the following linear payoffs for $x \in \mathbb R$:
\begin{eqnarray*}
f(x) = x - s, & \phi(x) = c + \lambda \abs{\delta}, & h(x)= a x, \\
g(x)= q - x, & \psi(x)= d + \gamma \abs{\delta}, & k(x) =  - b x, 
\end{eqnarray*}
with $s, \, c,\, \lambda, \, a, \, q, \, d, \, \gamma , \, b$ positive constants fulfilling
\begin{equation}\label{ass-ab}
a < \lambda \qquad \mbox{ and } \qquad b < \gamma .
\end{equation}
Hence, given an initial state $x$ and an impulse strategy $u= (\tau_n, \delta_n)_{n\geq 1}$, we define the controlled process $ X_t^{x;u}$ as
$$
X_t = X_t^{x;u} = x + \sigma W_t + \sum_{n:\tau_n \leq t} \delta_n, \qquad t \geq 0,
$$
where $W$ is a standard one dimensional Brownian motion and $ \sigma >0 $ is a fixed parameter. Moreover, we assume that the two players have the same discount factor $r_1 = r_2 = r$ such that 
 \begin{equation}\label{ass-r}
 1 - \lambda r > 0 \qquad \mbox{ and } \qquad 1 - br > 0.\end{equation}
The players' payoff functions are given by
\begin{align*}
J_1(x; u, \eta) &= \mathbb{E}_x \left[ \int_0^\eta e^{-r t} (X_t - s) dt - \sum_{n: \tau_n \leq \eta} e^{-r \tau_n} (c + \lambda \abs{\delta_n}) + a e^{- r \eta}  X_{\eta}\mathbbm{1}_{\{\eta < \infty\}}\right], \\
J_2(x; u, \eta) &= \mathbb{E}_x \left[ \int_0^\eta e^{-r t} (q - X_t ) dt + \sum_{n: \tau_n \leq \eta} e^{-r \tau_n} (d + \gamma \abs{\delta_n}) - b e^{- r \eta}X_{\eta}\mathbbm{1}_{\{\eta < \infty\}}\right] .
\end{align*}
Therefore in this game P1 can shift the state variable $X$ by intervening with impulses in order to keep it high enough, while paying some costs at each intervention time, until the end of the game, which is decided by P2. In addition to that P2, who want to keep $X$ low, might gain something each time P1 intervenes. At the end of the game, P1 (resp. P2) receives (resp. looses) some amount proportional to $X$. Hence, depending on whether her terminal payoff is high enough, P1 might want to end the game soon, by forcing P2 to do that. \medskip

Our goal is to find some Nash equilibrium by solving the QVI problem in \eqref{P1 intervention}-\eqref{P2 control}. More specifically, a heuristic analysis of the QVI system will help us finding a couple of quasi-explicit candidates $W_1, \, W_2$ for the equilibrium payoff functions of the game $V_1, \, V_2$. We recall the optimal impulse size and the intervention operators in this setting 
\begin{align*}
\{ \delta(x) \} &= \argmax_{\delta \in Z} \left\{ W_1(x+ \delta) - c - \lambda \abs{\delta}\right\},\\
\mathcal{M}W_1(x) &= W_1(x+ \delta(x)) - c - \lambda \abs{\delta(x)},\\
\mathcal{H}W_2(x) &= W_2 (x+ \delta(x)) + d + \gamma \abs{\delta(x)},
\end{align*}
together with the infinitesimal generator of the uncontrolled state variable
$$
\mathcal{A}W(x) = \frac{1}{2} \sigma^2 W''(x), \quad x \in \mathbb R.
$$
Before giving the QVI system in this case, let us introduce the continuation regions for both players
\begin{align*} 
\mathcal C_1 &= \{x \in \mathbb{R} : W_1(x + \delta(x)) - c - \lambda\abs{\delta(x)} < W_1(x) \},\\
\mathcal C_2 &= \{x \in \mathbb{R}: W_2(x) + bx = 0\},
\end{align*}
so that the respective intervention regions are given by $\mathcal C_i ^c$ for $i=1,2$. 
Now, the QVIs system becomes
$$
\begin{dcases}
W_1(x+ \delta(x)) - c - \lambda\abs{\delta(x)} - W_1(x) \leq 0, & x \in \mathbb R,\\
W_2(x) - bx \geq 0, \quad & x \in \mathbb R,\\
W_2(x+\delta(x)) + d + \gamma\abs{ \delta(x)} - W_2(x) = 0, & x \in \mathcal C_1 ^c, \\
W_1(x)  - ax = 0, & x \in \mathcal C_2 ^c ,\\
\max\left\{\frac{\sigma^2}{2}{W''_2}(x) - rW_2(x) + q - x,   -x b - W_2(x)\right\}=0, & x \in \mathcal C_1,\\
\max \left\{\frac{\sigma^2}{2}{W''_1}(x) - rW_1(x) + x - s,  (\mathcal{M}W_1 - W_1)(x)\right\} =0, & x \in \mathcal C_2.
\end{dcases}
$$
A first look at the system suggests the following representation for $W_1$ and $W_2$:
\begin{align}
W_1(x) &= \left\{\begin{array}{ll}
a x & x \in \mathcal C^c _2\\ 
\varphi_1 (x) & x \in \mathcal C_1 \cap \mathcal C_2 \\ 
\mathcal{M}W_1(x) & x \in \mathcal C_1 ^c \cap \mathcal C_2 \\ 
\end{array}\right.\label{ex:system1}\\
W_2(x) &= \left\{\begin{array}{ll}
 - bx & x \in \mathcal C_2 ^c\\ 
 \varphi_2 (x) & x \in \mathcal C_1 ^c \cap \mathcal C_2 \\
\mathcal{H}W_2(x) & x \in \mathcal C_1 ^c \cap \mathcal C_2, 
\end{array}\right.\label{ex:system2}
\end{align}
where $\varphi_1$ and $\varphi_2$ are solution to the ODEs
\begin{equation} \frac{1}{2}\sigma^2 \varphi_1''(x) - r \varphi_1(x) + x - s = 0, \quad \frac{1}{2}\sigma^2 \varphi_2''(x) - r \varphi_2(x) + q - x = 0. \label{ode}
\end{equation}
Hence, for each $ x \in \mathbb{R}$, we have:
\begin{equation} \varphi_1(x) = C_{11} e^{\theta x} + C_{12} e^{- \theta x} + \frac{x - s}{r}, \quad 
\varphi_2(x) = C_{21} e^{\theta x} + C_{22} e^{- \theta x} + \frac{q - x}{r},\label{sol-ode}
\end{equation}
where $C_{11}, \, C_{12}, \, C_{21}, \, C_{22}$ are real parameters and $\displaystyle{\theta = \sqrt{2r / \sigma^2}}$.

\subsection{An equilibrium with no simultaneous interventions}
In this subsection we push our heuristics further by focusing on a first type of Nash equilibrium, where simultaneous interventions are not allowed. We mean by that we are looking for an equilibrium of threshold type, where P1 intervenes each time $X$ falls below a certain level, say $\bar x_1$, in which case P1 applies an impulse moving the state variable towards an optimal level $x_1 ^*$ belonging to the continuation region of both players. On the other hand, P2 wait until $X$ is too high for him, i.e. until $X$ crosses some upper level, say $\bar x_2$, at which point P2 decides to stop the game. The heuristics will lead us to propose candidates for the equilibrium payoffs and related strategies, which will be then checked to be the correct ones subject to some additional conditions. Such additional conditions will be checked in some numerical examples.
 
\paragraph{Heuristics.}
Loosely speaking, since P1 is happy when $X$ is high while P2 prefers it to be low, we make the following ansatz about the continuation regions:
\begin{eqnarray*}
&\mathcal C_1^c = (-\infty, \bar x_1] & \mbox{(P1 intervenes)}, \\
&\mathcal C_1 \cap \mathcal C_2 = (\bar x_1, \bar x_2) & \mbox{(no one intervenes)} ,\\
&\mathcal C_2 ^c = [\bar x_2, \infty) & \mbox{(P2 intervenes)}.
\end{eqnarray*}
Hence, we can rewrite \eqref{ex:system1}-\eqref{ex:system2} as 
\begin{align}
W_1(x) &= \left\{\begin{array}{ll} 
a x, & x \in [\bar x_2, \infty) \\
\varphi_1 (x) , & x \in (\bar x_1, \bar x_2) \\
\mathcal{M}W_1(x), & x \in ( -\infty, \bar x_1]
\end{array}\right.\label{system1}\\
W_2(x) &= \left\{\begin{array}{ll}
 - bx , & x \in [\bar x_2, \infty)  \\
\varphi_2 (x), & x \in ( \bar x_1, \bar x_2)\\
\mathcal{H}W_2(x), & x \in ( -\infty, \bar x_1]
\end{array}\right.\label{system2}
\end{align}
Let us find more explicit expressions for the operators $\mathcal{M}W_1$ and $\mathcal{H}W_2$. In this example, it is natural to restrict the analysis to $\delta \geq 0$ since P1 prefers high values of $X^{x,u}$. Hence, whenever he intervenes he will always move the process $X$ to the right, so that
$$
\mathcal{M}W_1(x) = \sup_{\delta \geq 0} \left\{ W_1(x + \delta) - c - \lambda \abs{\delta} \right\} = \sup_{y \geq x} \left\{W_1(y) - c - \lambda(y-x)\right\}.
$$
Here we focus on the case where the maximum point belongs to $(\bar x_1, \bar x_2)$, in other words P1 does not force P2 to stop. In particular, we have $W_1(x_1^*) = \varphi(x_1^*)$ and  
$$
\varphi( x_1^*) = \max_{y \in ( \bar x_1, \, \bar x_2)} \left\{ \varphi(y) - \lambda y\right\}, \quad \mbox{ i.e. } \varphi_1'(x_1^*) = \lambda, \; \varphi_1''(x_1^*) \leq 0, \; \bar x_1 < x_1^* < \bar x_2.
$$
Therefore, we obtain
\[
\mathcal{M}W_1(x) = \varphi_1(x_1^*) - c - \lambda(x_1^* - x), \quad \mathcal{H}W_2(x) = \varphi_2(x_1^*) + d + \gamma(x_1^* - x).
\]
The parameters appearing in the expressions for $W_1$ and $W_2$ must be chosen so as to satisfy the regularity assumptions in the verification theorem, i.e.
\begin{eqnarray*}
&&W_1 \in C^2( (-\infty, \, \bar x_1] \cup (\bar x_1, \, \bar x_2)) \cap C^1(-\infty, \bar x_2) \cap C(\mathbb{R}),\\
&&W_2 \in C^2( (\bar x_1, \, \bar x_2) \cup ( \bar x_2, \, \infty)) \cap C^1(\bar x_1, \, \infty) \cap C(\mathbb{R}).
\end{eqnarray*}
We can summarize the description of our candidates for equilibrium payoffs in the following
\begin{ansatz} \label{qvi sol def}
Let $W_1$ and $W_2$ be as in \eqref{system1}-\eqref{system2} where the parameters involved $$( C_{11}, \, C_{12}, \, C_{21}, \, C_{22}, \, \bar x_1, \, \bar x_2, \, x_1^*)$$
satisfy the order condition
\begin{equation} \label{ordering}
\bar x_1 < x_1^* < \bar x_2 ,
\end{equation}
and the following equations
\begin{equation} \label{pasting system}
\begin{dcases}
\varphi'_1(x_1^*) = \lambda \quad \mbox{and} \quad \varphi''(x_1^*) \leq 0 & \mbox{(optimality of $x_1^*$)} \\
\varphi'_1(\bar x_1) = \lambda  & \mbox{($C^1$-pasting in $\bar x_1$)} \\
\varphi_2'(\bar x_2) = -b & \mbox{($C^1$-pasting in $\bar x_2$)} \\
\varphi_1(\bar x_1) = \varphi(x_1^*) - c - \lambda ( x_1^* - \bar x_1) & \mbox{($C^0$-pasting in $\bar x_1$)} \\
\varphi_1(\bar x_2) = a  \bar x_2 & \mbox{($C^0$-pasting in $\bar x_2$)} \\
\varphi_2(\bar x_1) = \varphi_2(x_1^*) + d + \gamma(x_1^* - \bar x_1) & \mbox{($C^0$-pasting in $\bar x_1$)}\\
\varphi_2(\bar x_2) = - b \bar x_2  &  \mbox{($C^0$-pasting in $\bar x_2$)}
\end{dcases}
\end{equation}
\end{ansatz}

\paragraph{Reparameterization.} 
We are going to conveniently reparameterize the equations above in order to reduce their complexity. Using the expressions in \eqref{sol-ode} we can rewrite \eqref{pasting system} as follows 
\begin{subnumcases}{}
\theta C_{11} e^{\theta x_1^*} - \theta C_{12} e^{ - \theta x_1^*} + \frac{1}{r} = \lambda  \label{(1)}\\
\theta C_{11} e^{\theta \bar x_1} - \theta C_{12} e^{ - \theta \bar x_1} + \frac{1}{r} = \lambda \label{(2)}  \\
\theta C_{21} e^{\theta \bar x_2} - \theta C_{22} e^{ - \theta \bar x_2} - \frac{1}{r} = -b  \label{(3)}\\
C_{11} e^{\theta \bar x_1} + C_{12} e^{ - \theta \bar x_1} + \frac{\bar x_1 - s}{r} = C_{11} e^{\theta x_1^*} + C_{12} e^{ - \theta x_1^*} + \frac{x_1^* - s}{r} - c - \lambda ( x_1^* - \bar x_1)  \label{(4)}\\
C_{11} e^{\theta \bar x_2} + C_{12} e^{ - \theta \bar x_2} + \frac{\bar x_2 - s}{r} = a  \bar x_2  \label{(5)}\\
C_{21} e^{\theta \bar x_1} + C_{22} e^{ - \theta \bar x_1} + \frac{q - \bar x_1}{r} = C_{21} e^{\theta x_1^*} + C_{22} e^{ - \theta x_1^*} + \frac{q - x_1^*}{r} + d + \gamma(x_1^* - \bar x_1) \label{(6)}\\
C_{21} e^{\theta \bar x_2} + C_{22} e^{ - \theta \bar x_2} + \frac{q - \bar x_2}{r} = -b \bar x_2  \label{(7)}
\end{subnumcases}
So, subtracting \eqref{(2)} to \eqref{(1)} we obtain 
\begin{equation*}
C_{11} = - \frac{1 - \lambda r }{r \theta} \frac{1}{e^{\theta x_1^*} + e^{\theta \bar x_1}}, \quad C_{12} = \frac{ 1 - \lambda r}{r \theta}\frac{e^{\theta(x_1^* + \bar x_1)}}{e^{\theta x_1^*} + e^{\theta \bar x_1}}.
\end{equation*}
Then, adding \eqref{(3)} to \eqref{(7)} we find
\begin{equation*}
C_{21} = \frac{e^{- \theta \bar x_2}}{2r} \left[ (1 - br) \left(\bar x_2 + \frac{1}{\theta}\right) - q \right], \quad C_{22} =  \frac{e^{\theta \bar x_2}}{2r} \left[ (1 - br ) \left(\bar x_2 - \frac{1}{\theta}\right) - q \right].
\end{equation*}
Hence, by substitution, we are reduced to solving the following sub-system
\begin{subnumcases}{}
-2\frac{1 - \lambda r}{r\theta} \frac{e^{\theta \bar x_1} - e^{\theta x_1^*}}{e^{\theta \bar x_1} + e^{\theta x_1^*}} + \frac{1 - \lambda r}{r} (\bar x_1 - x_1^*) + c =0 \label{(4)1}\\
- \frac{1 - \lambda r}{r\theta} \frac{e^{2 \theta \bar x_2}}{e^{\theta x_1^*} + e^{\theta \bar x_1}} + ((1 - ar)\bar x_2 - s)\frac{e^{\theta \bar x_2}}{r} + \frac{ 1 - \lambda r}{r\theta} \frac{e^{\theta(x_1^* + \bar x_1)}}{e^{\theta x_1^*} + e^{\theta \bar x_1}} = 0 \label{(5)1} \\
\frac{e^{- \theta \bar x_2}}{2r} \left[ (1 - br) \left(\bar x_2 + \frac{1}{\theta}\right) - q \right](e^{\theta \bar x_1} - e^{\theta x_1^*}) + \frac{e^{\theta \bar x_2}}{2r} \left[ (1 - br) \left(\bar x_2 - \frac{1}{\theta}\right) - q \right](e^{-\theta \bar x_1} - e^{-\theta x_1^*}) \nonumber\\
\quad  +\frac{1 - \gamma r}{r} (x_1^* - \bar x_1) - d = 0 \label{(6)1}
\end{subnumcases}
Now, the change of variable $z = e^{\theta(x_1^* - \bar x_1)}$ turns equation \eqref{(4)1} into the following
\begin{equation} \label{F(z)}
\ln z - 2\left(\frac{z -  1}{z + 1} \right) -  \frac{c r \theta}{1 - \lambda r } = 0,
\end{equation} 
which has a unique solution $\tilde z > 1$. Indeed, let $F(z) := \ln z - 2 (\frac{z - 1}{z + 1} ) - \frac{c r \theta}{1 - \lambda r}$ and observe that it satisfies $F'(z) > 0$ for all $z>1$. Moreover $z = e^{\theta(x_1^* - \bar x_1)} > 1$ due to the order condition \eqref{ordering}, $F(1) < 0$ and $\lim_{z \rightarrow +\infty} F(z) = +\infty$. Therefore, there is only one value $\tilde z$ such that $F(\tilde z) = 0$, which can be easily computed numerically. 

Now, in order to solve \eqref{(5)1} and \eqref{(6)1} we perform a second change of variable, $w = e^{\theta(\bar x_2 - \bar x_1)}$, leading to the following equations
\begin{subnumcases}{}
-\frac{1 - \lambda r }{r\theta} \frac{w^2 e^{\theta \bar x_1}}{\tilde z + 1} + \left( (1 - ar)\bar x_2 - s\right) \frac{e^{\theta \bar x_1}w}{r} + \frac{1 - \lambda r}{r\theta}\frac{e^{\theta x_1^*}}{\tilde z +1} = 0, &\label{(5)2} \\
 \frac{1 - \tilde z}{2rw}\left[(1 - br) \left(\bar x_2 + \frac{1}{\theta}\right) - q\right] + \frac{w(\tilde z - 1)}{2r\tilde z} \left[(1 - br) \left(\bar x_2 - \frac{1}{\theta}\right) - q\right] + \frac{1 - \gamma r}{\theta r} \ln \tilde z  - d = 0.& \quad\label{(6)2}
\end{subnumcases}
Notice that \eqref{(5)2} is linear in $\bar x_2$, hence it can be easily solved in terms of $\tilde z$ and $w$, to get
\begin{equation}\label{barx2}
\bar x_2 = \left(  \frac{1 - \lambda r}{\theta w}\frac{w^2 - \tilde z}{\tilde z+1} + s\right)\frac{1}{1 - ar}.
\end{equation}
Regarding \eqref{(6)2}, it can be rewritten as
\begin{eqnarray}
&& -\tilde z \left[(1 - br) \left(\bar x_2 + \frac{1}{\theta}\right) - q\right] + w^2\left[(1 - br) \left(\bar x_2 - \frac{1}{\theta}\right) - q\right] + \frac{2\tilde z w}{\tilde z - 1}\left(\frac{1 - \gamma r}{\theta}\ln \tilde z - rd \right) \nonumber \\
&&=w^4 \frac{(1 - br)(1 - \lambda r)}{\theta(1 - ar)(\tilde z +1)} + w^3\left[ (1 - br)\left(\frac{s}{1 - ar} - \frac{1}{\theta}\right) - q \right] +  2\tilde zw^2\left( \frac{1}{\tilde z - 1}\left(\frac{(1 - \gamma r)}{\theta}\ln \tilde z - rd\right) \right. \nonumber \\ 
&& \quad \left. - \frac{(1 - br)(1 - \lambda r)}{\theta(1 - ar)(\tilde z + 1)}\right)   + \tilde z w\left(q - (1 - br)\left(\frac{s}{1 - ar} + \frac{1}{\theta}\right)\right) + \frac{(1 - br)(1 - \lambda r) \tilde z^2}{\theta (1 - ar)(\tilde z +1)} = 0. \label{wequation}
\end{eqnarray}
The equation for $w$ above is a quartic equation for which explicit formulae for its roots are available. However, since they are quite cumbersome and not easy to use, we are going to solve it numerically, leaving the analysis for later. Once the two new parameters $\tilde z$ and $\tilde w$ are found, by solving numerically the respective equations above, the thresholds $\bar x_1, \bar x_2$ and the optimal level for P1, $x_1 ^*$, can be deduced automatically. It remains to check under which additional conditions such thresholds correspond to a Nash equilibrium of our original linear game. This will be done in the next paragraph.

\paragraph{Characterization of the equilibrium and verification.}

The next proposition summarizes our findings and establish the link between the solutions $\tilde z$ and $\tilde w$ to the equations above with the Nash equilibrium of threshold type we are looking for, provided some additional inequalities are fulfilled.
\begin{proposition} \label{propequi} Assume that there exists a solution $(\tilde z, \, \tilde w)$ to \eqref{F(z)}-\eqref{wequation} such that $1 < \tilde z < \tilde w$ and additionally
\begin{eqnarray}
&&
0 \le \frac{(1 - br)(1 - \lambda r)(\tilde w^2 - \tilde z)}{\theta \tilde w(1 - ar)(\tilde z + 1)} + \frac{1 - br}{1 - ar}s  - q < \frac{1 - br}{\theta},
\label{NE11}\\
&& \left(\frac{1 - br}{1 - ar}\left( \frac{(1 - \lambda r)(\tilde w^2 - \tilde z)}{\theta \tilde w(\tilde z + 1)} + s\right) - q\right)(\tilde w - 1)^2 + \frac{1 - br}{\theta}(1 + 2\tilde w \ln \tilde w - \tilde w^2) > 0.\label{NE12}\end{eqnarray} 
Then a Nash equilibrium for the game in Section \ref{game} exists and it is given by the pair $(u^*, \eta^*)$, where $u^* = (\tau_n, \, \delta_n)_{n \geq 1}$ is defined by
\[ \tau_n = \inf\left\{ t > \tau_{n -1}; X_t \in  (- \infty, \bar x_1]\right\} ,\quad  \delta_n = (x_1^* - x)\mathbf 1_{(-\infty, \bar x_1]}(x),\]
and
\[ \eta^* = \inf\{ t \geq 0: X_t \in [\bar x_2 , \, +\infty) \},\]
where the thresholds $\bar x_1, \, x_1^*$ and $\bar x_2$ satisfy
\begin{equation*}
x_1^* = \bar x_2 + \frac{\ln z -  \ln w}{\theta}, \quad \bar x_1 = \bar x_2 - \frac{\ln w}{\theta}, \quad \bar x_2 = \left(  \frac{1 - \lambda r}{\theta \tilde w}\frac{\tilde w^2 - \tilde z}{\tilde z+1} + s\right)\frac{1}{1 - ar}.
\end{equation*}
 Moreover, the functions $W_1, \, W_2$ in Ansatz \ref{qvi sol def} coincide with the equilibrium payoff functions $V_1, \, V_2$, i.e.
$$
V_1 \equiv W_1, \quad \mbox{ and } V_2 \equiv W_2.
$$
\end{proposition} 

\begin{proof} The proof consists in checking all the conditions needed to apply the Verification Theorem \eqref{oksendal-sulem verification}. 
First, notice that by construction the functions $W_1$ and $W_2$ satisfy all required regularity properties, i.e. $W_1$ and $W_2$ have polynomial growth, $W_1 \in C^2\!\left( (-\infty, \bar x_2)/ \{ \bar x_1\}\right) \cap C^1\left((-\infty, \bar x_2)\right) \cap C(\mathbb{R})$ and $W_2 \in C^2\left((\bar x_1, \infty )/ \{ \bar x_2\}\right) \cap C^1\left((\bar x_1, \infty)\right) \cap C(\mathbb{R})$.

Moreover Lemmas \ref{lemma w1 qvi} and \ref{lemma2} in the Appendix grant the optimality of the impulse $\delta(x)$, i.e.
\[\{ \delta(x) \} = \argmax_{\delta \in Z} \left\{ W_1(x+ \delta) - c - \lambda \abs{\delta}\right\}\]
together with the properties
\[\mathcal{M} W_1 - W_1 \leq 0, \quad W_2(x) +bx \geq 0, \quad x \in \mathbb R.\]
Next, we show that for all $x \in \{ \mathcal{M} W_1 - W_1 = 0 \} = ( -\infty, \, \bar x_1]$, we have $W_2(x) = \mathcal{H} W_2(x)$. Indeed, by definition of $\mathcal{H}W_2$ we have:
\begin{align*}
\mathcal{H}W_2(x)  &= W_2(x + \delta(x) ) + d + \gamma \abs{\delta(x)} = W_2(x_1^*) + d + \gamma (x_1^* - x) \\
&= \varphi_2(x_1^*) + d + \gamma (x_1^* - x) = W_2(x) \qquad \forall \, x \in ( -\infty, \, \bar x_1] .
\end{align*}
Now, let $x \in \{\mathcal{M} W_1 - W_1 < 0\}$. We have to prove that 
$$
\max\{ \mathcal{A} W_2(x) - rW_2(x) + q - x, \, -bx - W_2(x) \} = 0.
$$   
Since $\{\mathcal{M} W_1 - W_1 < 0\} =  (\bar x_1, \, \infty)$, we can consider two separate cases. In $(\bar x_1, \, \bar x_2)$ we have $-bx - W_2(x) <0$ and 
$$
\mathcal{A} W_2(x) - r W_2(x) + q - x = \mathcal{A} \varphi_2(x) - r \varphi_2(x) + q - x = 0 
$$ 
since $\varphi_2$ is solution to the ODE \eqref{ode}. On the other hand, in $[\bar x_2, \, \infty)$ we know that $-bx = W_2(x)$, then we have to check that $\mathcal{A}W_2(x) - rW_2(x) + q - x \leq 0$ for all $x \in [\bar x_2, \, \infty)$. First, notice that $W_2(x) = - bx$ and $\mathcal{A}W_2(x) = 0$. Hence, we are reduced to checking the inequality
\begin{equation} \label{ 1 - br greater}
\mathcal{A}W_2(x) - r W_2(x) + q - x = brx + q - x = q - (1 - br)x \leq 0.
\end{equation}
Since by assumption $1 - br > 0$, the function $x \mapsto q - ( 1 - br) x$ is decreasing, so we just need to check whether the inequality holds in $\bar x_2$, i.e. $(1 - br)\bar x_2 - q\geq 0 $ which is satisfied by \eqref{NE11}.

To conclude our verification that the candidate equilibrium payoffs satisfy the QVI system, we are left with checking that $-bx - W_2(x) = 0$ implies $W_1(x) = ax$, and that, on the other side, $-bx - W_2(x) < 0$ implies
\[ \max \{\mathcal{A} W_1(x) - r W_1(x) + x - s, \mathcal{M} W_1 (x) - W_1(x)\} = 0.\]
Now, the first implication holds by definition, while the second one boils down to proving 
$$
 \max \{ \mathcal{A}W_1(x) - rW_2(x) +x - s, \, \mathcal{M}W_1(x) - W_1(x)\} = 0, \quad  x \in ( - \infty,  \bar x_2).
$$ 
For $x \in (\bar x_1, \, \bar x_2)$ we have $\mathcal{M} W_1(x) - W_1(x) < 0$ and, as before,
$$
\mathcal{A}W_1(x) - rW_1(x) + x - s = \mathcal{A}\varphi_1(x) - r \varphi_1(x) + x - s = 0
$$
 as $\varphi_1$ is solution to the ODE \eqref{ode}. For $x \in (-\infty, \bar x_1]$ we know that $ \mathcal{M}W_1(x) - W_1(x) = 0$ and therefore we have to check that 
$$
 \mathcal{A}W_1(x) - rW_1(x) + x - s \leq 0 ,\quad  x \in ( -\infty, \bar x_1].
 $$
 To do that, recall first that $W_1(x) = \varphi_1(x_1^*) - c - \lambda(x_1^* - x)$ and $\mathcal{A}W_1(x) = 0$, which gives
 \begin{align*}
\mathcal{A}W_1(x) - rW_1(x) + x - s &= -r \varphi_1 (x_1^*) + rc + r \lambda (x_1^*\! -\! x) \!+\! x\!-\! s \\
&= - r ( \varphi_1(\bar x_1) \!+\! c \!+\! \lambda(  x_1^* \!-\! \bar x_1)) \!+\! rc \! + \!r\lambda(x_1^* \!-\! x)\! +\! x\! -\! s \!  \\
 & = - r \varphi_1(\bar x_1) - r \lambda (x - \bar x_1) + x -s ,
 \end{align*}
 since $ \varphi_1(\bar x_1) = \varphi_1(x_1^*) - c - \lambda (x_1^* - \bar x_1) $. Notice that, since by assumption $ 1 - \lambda r > 0$, the function $x \mapsto - r \varphi_1(\bar x_1) - r \lambda (x - \bar x_1) + x -s $ is increasing in $x$. As a result, we only need to prove that the desired inequality holds for $ x = \bar x_1$, i.e.
 $$
 -r \varphi_1(\bar x_1) + \bar x_1 - s \leq 0 ,
 $$
 which is verified since $\mathcal{A}\varphi_1(\bar x_1) - r \varphi_1(\bar x_1) + \bar x_1 - s = 0$ and $\mathcal{A} \varphi_1(\bar x_1) = r \varphi_1(\bar x_1) - \bar x_1 + s \geq 0$, due to $\varphi_1''(\bar x_1) \geq 0$.

To finish the proof, we check that equilibrium strategies are $x$-admissible for every $x \in \mathbb{R}$. By construction, the controlled process never exits from $(\bar x_1, \, \bar x_2) \cup \{ x \}$, so that $\sup_{t \geq 0} \abs{X_t} \in L^p(\Omega)$ holds. It is easy to check that all the other conditions are satisfied provided we show the following:
\begin{equation} \label{integrability intervention cost}
\mathbb{E}_x \left[ \sum_{k \geq 1} e^{-r \tau_k} ( c+ \lambda \abs{\delta_k}) \right] < +\infty  .
\end{equation}
To start, let us assume that the initial state $x$ is $x_1^*$. The idea is to write $\tau_k$ as a sum of independent and identically distributed copies of some exit time (as in the proof of Proposition 4.7 in \cite{aid2016nonzero}). Denote by $\mu$ the exit time of the process $x_1^* + \sigma W$ from $(\bar x_1, \, \bar x_2)$ where $W$ is a one-dimensional Brownian motion. Then each time $\tau_k$ can be decomposed as $\tau_k =\sum_{l \geq 1}^k \zeta_l$ where $\zeta_l$ are i.i.d. random variables with the same law as $\mu$.  We can now show \eqref{integrability intervention cost}. As $ \delta_k = \delta_1 =  x_1^* - \bar x_1 $ for all $k \ge 1$, we have 
\begin{align*}
\mathbb{E}_{x_1^*} \left[ \sum_{k \geq 1} e^{-r \tau_k}(c + \lambda \abs{\delta_k})\right] &\leq (c + \lambda \delta_1 ) \mathbb{E}_{x_1^*} \left[ \sum_{k\geq 1} e^{-r \tau_1}\right] \\
&= (c + \lambda \delta_1 ) \mathbb{E}_{x_1^*} \left[ \sum_{k \geq 1}  e^{ -r \sum_{l=1}^k \zeta_l} \right] \\
&=  (c + \lambda \delta_1 ) \mathbb{E}_{x_1^*} \left[ \sum_{k \geq 1} \Pi_{l\geq 1}^k e ^{- r \zeta_l} \right]
\end{align*}
and, by the Fubini-Tonelli theorem and the independence of $(\zeta_l)_{l\ge 1}$, we get 
$$
\sum_{k \geq 1} \Pi_{l \geq 1}^k \mathbb{E}_{x_1^*} \left[ e^{-r\zeta_l} \right] \leq \sum_{k \geq 1} \left( \mathbb{E}_{x_1^*} \left[ e^{- r\mu} \right] \right)^k ,
$$
which is a convergent geometric series, since $\mu \geq 0$. Then, for any $x \in (\bar x_1, \, \bar x_2)$ same arguments hold whereas, when $x \in [\bar x_2, \, +\infty)$, P2 stops as soon as the game starts and, as a consequence, P1 cannot apply any impulse, hence, the condition is satisfied. Finally, if $x \in (-\infty, \, \bar x_1]$ we have  
$$
\mathbb{E}_x \left[ \sum_{k \geq 1} e^{-r \tau_k} ( c+ \lambda \abs{\delta_k}) \right] = c + \lambda \abs{x_1^* - x} + \mathbb{E}_{x_1^*} \left[ \sum_{k \geq 1} e^{-r \tau_k} ( c+ \lambda \abs{\delta_k}) \right] < +\infty .
$$
since $\sup_{t \geq 0} \abs{X_t} \in L^p(\Omega)$.
\qed
\end{proof}

\subsection{An equilibrium where the controller activates the stopper}
We turn now to another kind of Nash equilibrium, where P1 behaves similarly as in the previous type with the main difference that this time when the state variable $X$ falls below a given threshold, he will intervene and send $X$ directly to the stopping region of P2, hence forcing him to stop the game instantaneously. In particular, this would be an equilibrium in which the two players act at the same time. The approach we are going to use to characterize such an equilibrium follows the same steps as in the previous subsection.  
\paragraph{Heuristics.}
We start with some heuristics leading us to formulate a conjecture on the equations the thresholds characterizing this equilibrium should reasonably satisfy. Arguing as before, we expect the candidates for equilibrium payoffs to be of the following type \eqref{ex:system1}-\eqref{ex:system2} as 
\begin{eqnarray}
W_1(x) &=& \left\{\begin{array}{ll} 
a x & \mbox{in} [\bar x_2, \infty) \\
\varphi_1 (x) & \mbox{in} (\bar x_1, \bar x_2) \\
\mathcal{M}W_1(x) & \mbox{in} ( -\infty, \bar x_1]
\end{array}\right.\label{system3}\\
W_2(x) &=& \left\{\begin{array}{ll}
 - bx & \mbox{in} [\bar x_2, \infty)  \\
\varphi_2 (x) & \mbox{in} ( \bar x_1, \bar x_2)\\
\mathcal{H}W_2(x) & \mbox{in} ( -\infty, \bar x_1]
\end{array}\right.\label{system4}
\end{eqnarray}
for suitable thresholds $\bar x_i$, $i=1,2$.

Now, according to the type of equilibrium we want to identify, we investigate the case in which the maximum point of the function $y \mapsto W_1(y) - \lambda y $ belongs to $ [ \bar x_2, \, \infty ) $, meaning that when P1 intervenes he is applying an optimal impulse moving the state variable to the stopping region of her competitor. Thus in this case we have
$$
 \mathcal{M} W_1 (x) = \sup_{y \geq \bar x_2} (ay - \lambda y) .
$$
Therefore, we have the following scenarios:
\begin{itemize}
\item if $ a > \lambda \, \Rightarrow \, x_1^* \rightarrow \infty $;
\item if $ a = \lambda \, \Rightarrow \, x_1^* $ could be any $x \geq \bar x_2$;
\item if $ a < \lambda \, \Rightarrow \, x_1^* = \bar x_2$.
\end{itemize}
Clearly, the only interesting case is $ a < \lambda$,  so that $x_1^* = \bar x_2$. As a consequence, this type of equilibrium will be characterized only by two thresholds. Similarly as in the previous subsection, we are going to  characterize the parameters $( C_{11}, C_{12},  C_{21}, C_{22})$ and the thresholds $(\bar x_1, \bar x_2 )$ by exploiting the smooth pasting conditions coming from the regularity assumptions postulated in Theorem \ref{oksendal-sulem verification}. By doing so, we obtain
\begin{equation} \label{2pasting system}
\begin{dcases}
\varphi'_1(\bar x_1) = \lambda  & \mbox{($C^1$-pasting in $\bar x_1$)} \\
\varphi_1(\bar x_2) = a \bar x_2  & \mbox{($C^0$-pasting in $\bar x_2$)} \\
\varphi_1(\bar x_1) = a  \bar x_2 - c - \lambda (\bar x_2 - \bar x_1) & \mbox{($C^0$-pasting in $\bar x_1$)} \\
\varphi_2'(\bar x_2) = -b & \mbox{($C^1$-pasting in $\bar x_2$)} \\
\varphi_2(\bar x_2) =  - b \bar x_2 & \mbox{($C^0$-pasting in $\bar x_2$)}\\
\varphi_2(\bar x_1) = - b \bar x_2 + d + \gamma (\bar x_2 - \bar x_1)  &  \mbox{($C^0$-pasting in $\bar x_1$)}
\end{dcases}
\end{equation}
together with the order condition $ \bar x_1 < \bar x_2 $.

\paragraph{Reparameterization.}
We first rewrite \eqref{2pasting system} as
\begin{subnumcases}{}
\theta C_{11} e^{\theta \bar x_1 } - \theta C_{12} e^{-\theta \bar x_1} + \frac{1}{r} = \lambda  \label{1)}  \\
\theta C_{21} e^{\theta \bar x_2} - \theta C_{22} e^{ - \theta \bar x_2} - \frac{1}{r} = - b \label{2)}\\
C_{11} e^{\theta \bar x_2} + C_{12} e^{-\theta \bar x_2} + \frac{\bar x_2 - s}{r} = a \bar x_2 \label{3)} \\
C_{11} e^{\theta \bar x_1} + C_{12} e^{-\theta \bar x_1} + \frac{\bar x_1 - s}{r} = (a - \lambda) \bar x_2 + \lambda \bar x_1 - c  \label{4)}\\
 C_{21} e^{\theta \bar x_2} + C_{22} e^{ - \theta \bar x_2} + \frac{q - \bar x_2}{r} = - b \bar x_2 \label{5)}\\
 C_{21} e^{\theta \bar x_1} +  C_{22} e^{ - \theta \bar x_1} + \frac{q - \bar x_1}{r} = (\gamma - b) \bar x_2 + d - \gamma \bar x_1 \label{6)}
\end{subnumcases} 

Then, dividing \eqref{1)} by $\theta$ and adding it to \eqref{4)}, we obtain a linear equation in $C_{11}$ that can be solved giving
\begin{eqnarray}
C_{11} = \frac{e^{-\theta \bar x_1}}{2} \left[(a - \lambda) \bar x_2 - \left(\bar x_1 + \frac{1}{\theta}\right)\frac{1 - \lambda r}{r} - c + \frac{s}{r}\right] , \label{C11 i}
\end{eqnarray}
and consequently
\begin{equation} \label{C12 i}
C_{12} = \frac{e^{\theta \bar x_1}}{2} \left[ ( a - \lambda) \bar x_2 - \left( \bar x_1 - \frac{1}{\theta}\right) \frac{1 - \lambda r}{r} - c + \frac{s}{r} \right].
\end{equation}
A similar manipulation of equations \eqref{2)} and \eqref{5)} yields
\begin{align}
C_{21} & = \frac{e^{-\theta \bar x_2}}{2r} \left[ (1 - br) \left(\bar x_2 + \frac{1}{\theta} \right) - q \right] \label{C21 i},\\ 
C_{22} & = \frac{e^{\theta \bar x_2}}{2r} \left[ (1 - br) \left(\bar x_2 - \frac{1}{\theta} \right) - q \right]. \label{C22 i}
\end{align}
At this point, plugging \eqref{C11 i} and \eqref{C12 i} in \eqref{3)} we obtain 
\begin{eqnarray*}
&&\frac{e^{\theta ( \bar x_2 - \bar x_1)}}{2} \left[(a - \lambda) \bar x_2 - \left(\bar x_1 + \frac{1}{\theta}\right)\frac{1 - \lambda r}{r} - c + \frac{s}{r}\right]  + \frac{e^{-\theta( \bar x_2 - \bar x_1)}}{2} \\
&& \qquad \times \left[ ( a - \lambda) \bar x_2 - \left( \bar x_1 - \frac{1}{\theta}\right) \frac{1 - \lambda r}{r} - c + \frac{s}{r} \right] + \frac{1 - ar}{r}\bar x_2 - \frac{s}{r} = 0 
\end{eqnarray*}
which, noting that $\bar x_1 = \bar x_2 - \frac{\ln w}{\theta}$ and applying the change of variable $w = e^{\theta( \bar x_2 - \bar x_1)}$, can be rewritten as
\begin{eqnarray*}
&&w \left[\frac{(1 - \lambda r)(\ln w -1)}{r\theta} - \frac{1 - ar}{r}\bar x_2 - c + \frac{s}{r}\right]  + \frac{1}{w}  \left[ \frac{(1 - \lambda r)(\ln w + 1)}{r\theta} - \frac{1 - ar}{r}\bar x_2 - c + \frac{s}{r} \right]\\
&& \quad +  2\frac{(1 - ar)\bar x_2 - s}{r}= 0.
\end{eqnarray*}
This is a linear equation in $\bar x_2$, yielding
\begin{equation} \label{x2 bar 1}
\bar x_2 = \frac{ (1 - \lambda r)((\ln w - 1)w^2 + \ln w + 1) - cr\theta(w^2 + 1)}{\theta(1 - ar)(w - 1)^2} + \frac{s}{1 - ar}. 
\end{equation}
Proceeding analogously with \eqref{6)}, we obtain the following alternative expression for $\bar x_2$
 \begin{equation} \label{x2 bar 2}
 \bar x_2 =   \frac{q}{1 - br} + \frac{w + 1}{\theta(w - 1)} + \frac{2(\theta rd - (1 - \gamma r) \ln w)w}{\theta ( 1 - b r)(w - 1)^2}. 
 \end{equation}
Then, by equating \eqref{x2 bar 1} to \eqref{x2 bar 2}, we obtain an equation in $w$:
\begin{align}
G(w) := & \frac{ (1 - \lambda r)((\ln w - 1)w^2 + \ln w + 1) - cr\theta(w^2 + 1)}{\theta(1 - ar)(w - 1)^2} + \frac{s}{1 - ar} \nonumber\\
&  - \frac{q}{1 - br} - \frac{w + 1}{\theta(w - 1)} - \frac{2(\theta rd - (1 - \gamma r) \ln w)w}{\theta ( 1 - b r)(w - 1)^2} = 0 \label{G(w) second case}
\end{align}
which has at least a solution, say $\hat w > 1$,
due to $\lim_{w \rightarrow + \infty} G(w) = + \infty$ and  $\lim_{w \rightarrow 1} G(w) = - \infty$. The first limit follows from the highest order term, $w^2\ln w$, being multiplied by $\frac{1 - \lambda r}{1 - ar} > 0$ (cf. \eqref{ass-r}). On the other hand, the second limit follows from \eqref{ass-ab}:
$$
\lim_{w \rightarrow 1} G(w) =\lim_{w \rightarrow 1} \frac{1}{(w-1)^2}\left[-\frac{2cr}{1 - ar} - \frac{2rd}{1 - br}\right] = - \infty.
$$

\paragraph{Characterization of the equilibrium and verification.}

The next proposition summarizes our characterization of this Nash equilibrium in terms of only one parameter, $\hat w$, provided some further conditions, that will be checked numerically in the next subsection.
\begin{proposition} \label{propequi2} Assume that there exists $\hat w$ solution to \eqref{G(w) second case} such that 
\begin{eqnarray}
&& (1 - \lambda r)(\hat w - \hat w\ln \hat w - 1) + cr\theta \hat w > 0, \label{NE21}\\
&& 0 \le (1 - br)(\hat w^2 - 1) + 2(\theta r d - ( 1 - \gamma r) \ln \hat w) \hat w < (1 - br)(\hat w - 1)^2 .
\label{NE22}
\end{eqnarray} 
Then a Nash equilibrium for the game in Section \ref{game} exists and it is given by the strategies $(u^*, \eta^*)$, with $u^* = (\tau_n, \delta_n)_{n \geq 1}$ defined by
\begin{equation*}
\tau_n = \inf\left\{ t > \tau_{n -1}; X_t \in  (- \infty, \bar x_1]\right\}, \quad \delta_n = \left(\bar x_2 - x\right) \mathbf 1_{(-\infty, \bar x_1]}(x) 
\end{equation*}
and
\[
\eta^* = \inf\{ t \geq 0: X_t \in [\bar x_2 , \, +\infty) \} ,\]
where the thresholds satisfy
\begin{equation*}
 \bar x_1 = \bar x_2 - \frac{\ln \hat w}{\theta}, \quad \bar x_2 =  \frac{q}{1 - br} + \frac{\hat w + 1}{\theta(\hat w - 1)} + \frac{2(\theta rd - (1 - \gamma r) \ln \hat w)\hat w}{\theta ( 1 - b r)(\hat w - 1)^2}.
\end{equation*}
 Moreover, the functions $W_1, \, W_2$ in Ansatz \ref{qvi sol def} coincide with the equilibrium payoff functions $V_1, \, V_2$, i.e.
$$
V_1 \equiv W_1 \quad \text{and}\quad V_2 \equiv W_2 .
$$
\end{proposition} 
\begin{proof} We proceed as for the previous equilibrium, by checking all the conditions necessary to apply the verification theorem. 
First of all, the functions $W_1,W_2$ satisfy by construction all required regularity properties, i.e. $W_1 \in C^2\!\left( (-\infty, \bar x_2)/ \{ \bar x_1\}\right) \cap C^1\left((-\infty, \bar x_2)\right) \cap C(\mathbb{R})$, $W_2 \in C^2\left((\bar x_1, \infty )/ \{ \bar x_2\}\right) \cap C^1\left((\bar x_1, \infty)\right) \cap C(\mathbb{R})$ and both have at most polynomial growth.

Next, Lemmas \ref{lemma2 w1 qvi} and \ref{lemma22} give
\[\{ \delta(x) \} = \argmax_{\delta \in Z} \left\{ W_1(x+ \delta) - c - \lambda \abs{\delta}\right\}\]
together with
\[ \mathcal{M} W_1(x) - W_1(x) \leq 0, \quad W_2(x) +bx \geq 0,\]
for all $x \in \mathbb R$. 
Let $x \in \{ \mathcal{M} W_1 - W_1 = 0 \} = ( -\infty, \, \bar x_1]$. By definition of $\mathcal{H}W_2$ we have:
\begin{align*}
\mathcal{H}W_2(x)  &= W_2(x + \delta(x) ) + d + \gamma \abs{\delta(x)} = W_2(\bar x_2) + d + \gamma (\bar x_2 - x) \\
&= -b \bar x_2 + d + \gamma (\bar x_2 - x) = W_2(x). 
\end{align*}
Now, in order to prove that 
$$
\max\{ \mathcal{A} W_2(x) - rW_2(x) + q - x, \, -bx - W_2(x) \} = 0 ,\quad x \in (\bar x_1, \, \infty),
$$
we consider two separate cases as for the previous equilibrium. First, for $x \in (\bar x_1, \, \bar x_2)$, we have $-bx - W_2(x) <0$ and 
$$
\mathcal{A} W_2(x) - r W_2(x) + q - x = \mathcal{A} \varphi_2(x) - r \varphi_2(x) + q - x = 0 
$$ 
since $\varphi_2$ is solution to the ODE \eqref{ode}, so the maximum between the two terms is zero. Second, we know that $-bx = W_2(x)$ for $x \in [\bar x_2, \, \infty)$, then we have to check that $\mathcal{A}W_2(x) - rW_2(x) + q - x \leq 0$ for any $x \in [\bar x_2, \, \infty)$. Since $\mathcal{A}W_2(x) = 0$, we are reduced to verify the inequality
\begin{equation} \label{ 1 - br greater 2}
\mathcal{A}W_2(x) - r W_2(x) + q - x = brx + q - x = q - (1 - br)x \leq 0.
\end{equation}
Given that $x \mapsto q - ( 1 - br) x$ is decreasing due to $1 - br > 0$, it suffices to show the inequality above at the point $\bar x_2$, i.e. $(1 - br)\bar x_2 - q\geq 0 $, which is implied by \eqref{NE22}.\\
To complete the verification that $W_1,W_2$ are solutions to the QVI system, we show that in $-bx - W_2(x) = 0$ implies $ W_1(x) = ax$ and that $-bx - W_2(x) < 0$ yields $\max \{\mathcal{A} W_1(x) - r W_1(x) + x - s, \mathcal{M} W_1(x) - W_1(x)\} = 0$.
The first implication holds by definition. For the second one, we have to prove 
$$
 \max \{ \mathcal{A}W_1(x) - rW_2(x) +x - s, \, \mathcal{M}W_1(x) - W_1(x)\} = 0, \quad x \in ( - \infty, \, \bar x_2).
$$ 
For $x \in (\bar x_1, \, \bar x_2)$ we have $\mathcal{M} W_1(x) - W_1(x) < 0$ and as before
$$
\mathcal{A}W_1(x) - rW_1(x) + x - s = \mathcal{A}\varphi_1(x) - r \varphi_1(x) + x - s = 0
$$
 as $\varphi_1$ is solution to the ODE \eqref{ode}. For any $x \in (-\infty, \bar x_1]$ we know that $ \mathcal{M}W_1(x) - W_1(x) = 0$, hence, we have to check that 
$$
 \mathcal{A}W_1(x) - rW_1(x) + x - s = (1 - \lambda r)x + cr - s - (a - \lambda)r\bar x_2 \leq 0, \quad x \in ( -\infty, \, \bar x_1].
 $$
To do so, we notice that, since by assumption $ 1 - \lambda r > 0$, the function $x \mapsto (1 - \lambda r)x + cr - s - (a - \lambda)r\bar x_2$ is increasing in $x$. Therefore, we only need to prove that the desired inequality for $ x = \bar x_1$, i.e.
 $$
(1 - ar)\bar x_2 - \frac{ 1 - \lambda r}{\theta}\ln w + cr - s \leq 0 ,
 $$
 which is given by Lemma \ref{lemma2 w1 qvi}. Finally, the optimal strategies are $x$-admissible for every $x \in \mathbb{R}$. Indeed, by construction, the controlled process never exits from $(\bar x_1, \bar x_2) \cup \{ x \}$, so that $\sup_{t \geq 0} \abs{X_t} \in L^p(\Omega)$ for all $p \ge 1$. It is easy to check that all the other conditions are satisfied as in the first type of equilibrium. 
\qed
\end{proof}

\subsection{Numerical experiments} \label{numerics}

In this section we will give some numerical illustrations of the equilibrium payoff functions and a selection of comparative statics regarding the two types of Nash equilibria identified in the previous subsections.\footnote{The numerical results in this section were obtained using R, rootSolve package.} 
It is useful to remember that in order for the solutions to the QVI system to be Nash equilibria of one of the two types, they have to satisfy either \eqref{NE11}-\eqref{NE12} or \eqref{NE21}-\eqref{NE22}. Before we start, let us recall the meaning of the parameters involved:
\begin{itemize}
\item $s$ and $q$ might be interpreted as exogenous costs and gains respectively. Note that P1's running payoff $f(x) = x -s$, hence, in order to make profit P1 needs $x$ to be greater than $s$, which can fairly be considered as P1's expense, an analogous reasoning applies for P2, but in the opposite direction since $g(x) = q - x$;
\item $a$ and $b$ can be considered as terminal payoff sensitivity to the underlying process, $X_t$, as we have $h(x) = ax$ and $k(x) = -bx$ respectively;
\item at each intervention time P1 faces a fixed cost, $c$, while P2 receives a fixed gain, $d$; 
\item moreover,  $\lambda$ is P1's proportional cost parameter, while $\gamma$ is P2's proportional gain parameter;
\item finally, $r$ is the discount rate, the same for both players, and $\sigma$ is the volatility of the state variable.
\end{itemize}

\paragraph{Equilibrium with no simultaneous interventions.} In order to fulfill \eqref{NE11}-\eqref{NE12}, we can observe that both inequalities are satisfied for high enough values of $\tilde w$. It is possible to show via graphical analysis  that $\tilde w$, solution to \eqref{wequation}, is decreasing in $a, b, s$ and increasing in $c, d, q, \lambda$ and $\gamma$. Therefore, we have chosen small values of $a, b$ and $s$ to obtain the first equilibrium, Scenario A, whereas for Scenario B we have looked for higher values and increased $q$ and $d$ in order to find an equilibrium. 
\begin{center} \begin{tabular}{ c | c | c | c | c | c |c | c | c | c | c || c | c | c} 
   & $r$ & $\sigma$ & $c$ & $d$ & $\lambda$ & $\gamma$ & $a$ & $b$ & $s$ & $q$ & $\bar x_1$ & $x_1^*$ & $\bar x_2$\\
 \hline Scenario A & 0.01 & $5$ & 500 & 100 & 20 & 40 & 0 & 0 & 1 & 5 & -31.11 & 16.95 & 34.84\\
  Scenario B & 0.01 & $1.5$ & 50 & 150 & 10 & 15 & 2 & 8 & 10 & 10 & 4.95 & 14.26 & 18.18.\\
\end{tabular}\end{center}
where $\bar x_1$, $x_1^*$ and $\bar x_2$ are as in Proposition \ref{propequi}.
\begin{figure}[H]
\centering
\captionsetup[subfigure]{labelformat=simple}
\caption{Type I Equilibria}
\begin{subfigure}[H]{0.49\textwidth}
\includegraphics[width=\textwidth]{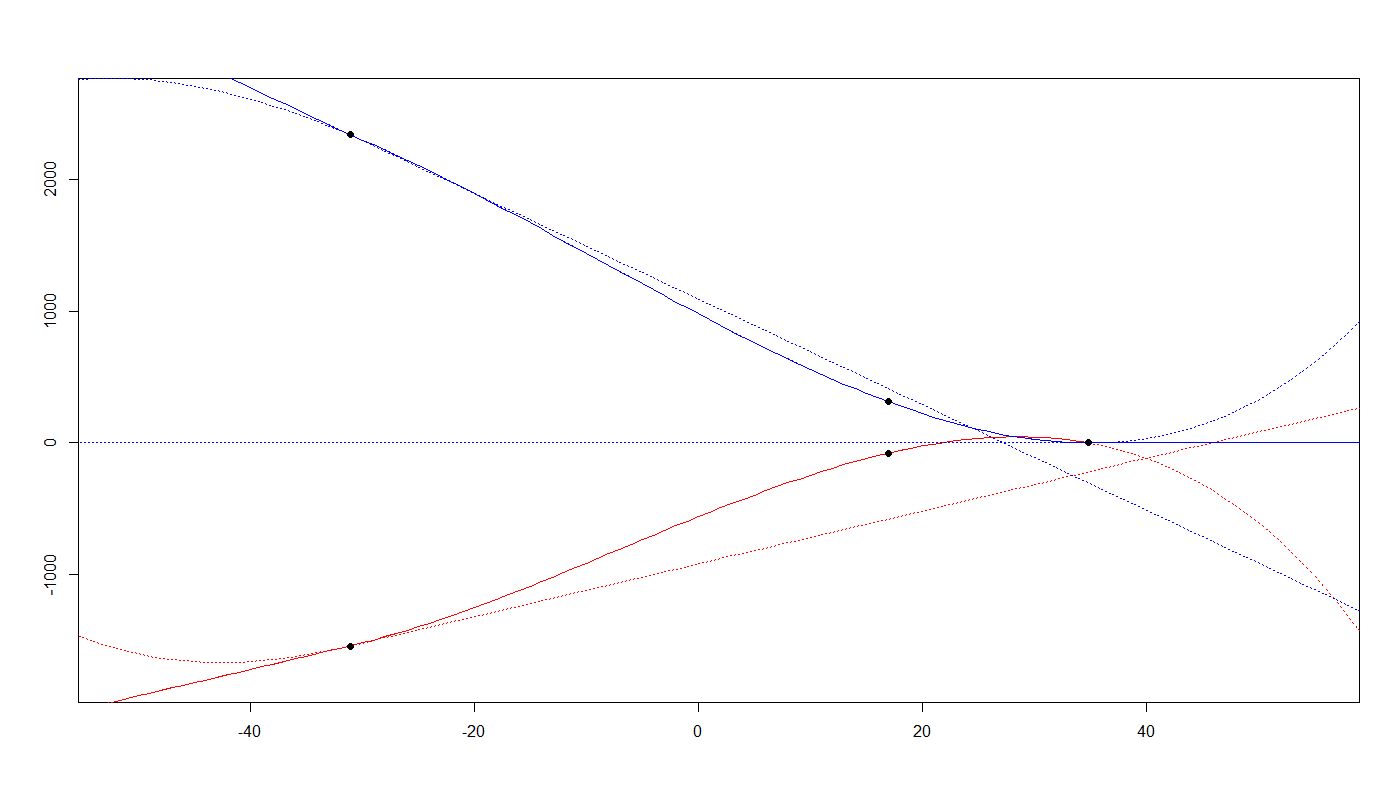}
\caption{$ x \!\mapsto\! V_1(x)$ in red, $ x \!\mapsto\! V_2(x)$ in blue for Scenario A \label{ne11graph}}
\end{subfigure}
\begin{subfigure}[H]{0.49\textwidth}
\includegraphics[width=\textwidth]{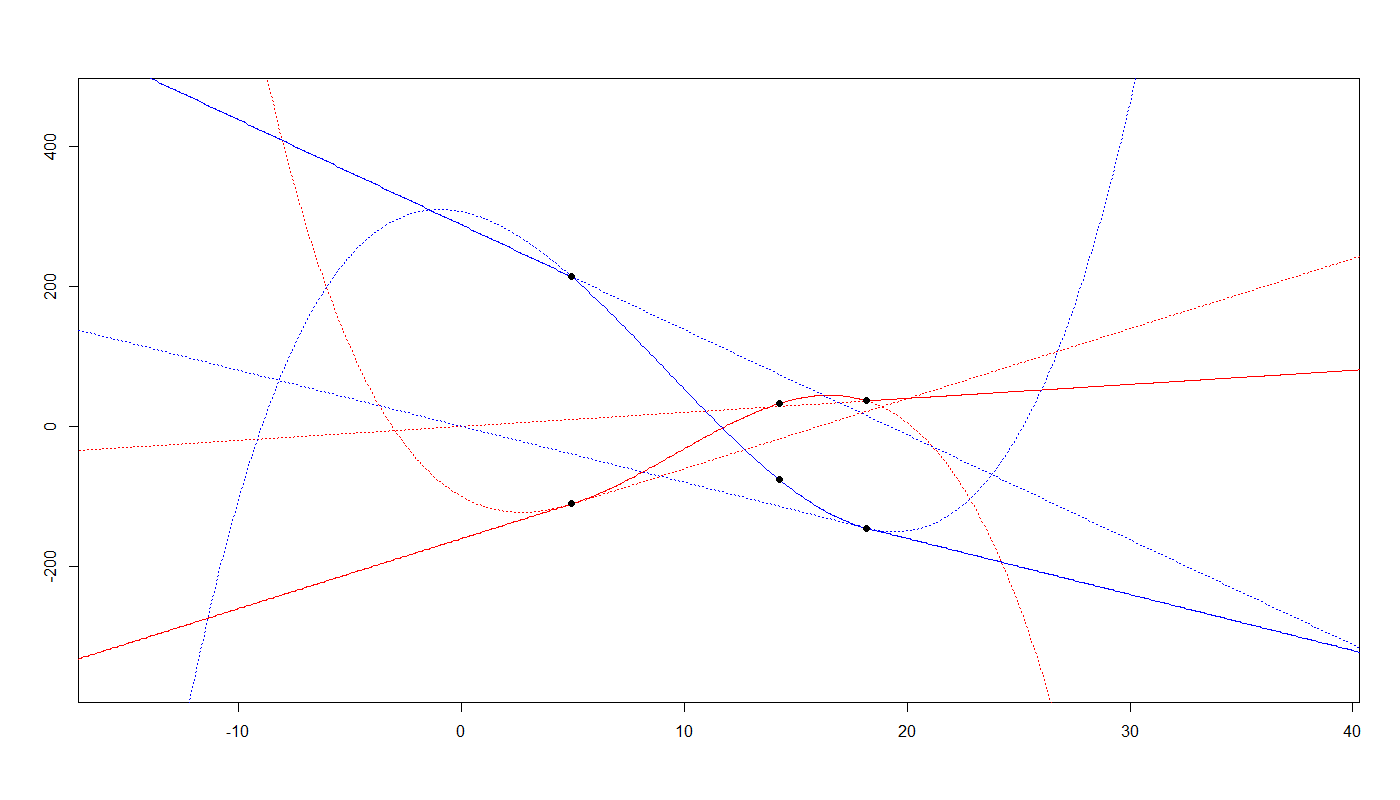}
\caption{$ x \!\mapsto\! V_1(x)$ in red, $ x \!\mapsto\! V_2(x)$ in blue for Scenario B \label{ne12graph}}
\end{subfigure}
\begin{subfigure}[H]{0.49\textwidth}
\includegraphics[width=\textwidth]{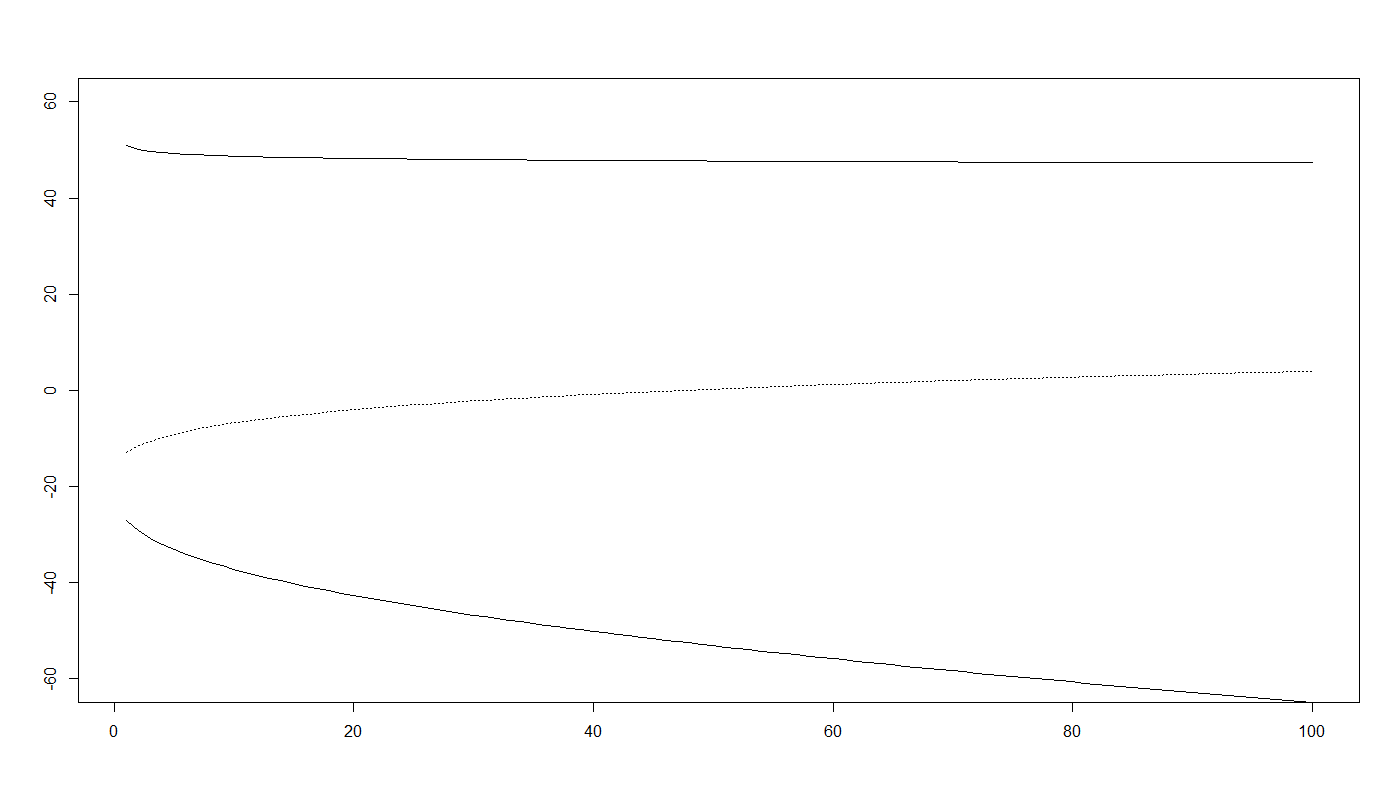}
\caption{$ c \mapsto \bar x_1, x_1^*, \bar x_2$ for Scenario B \label{cs1c} }
\end{subfigure}
\begin{subfigure}[H]{0.49\textwidth}
\includegraphics[width=\textwidth]{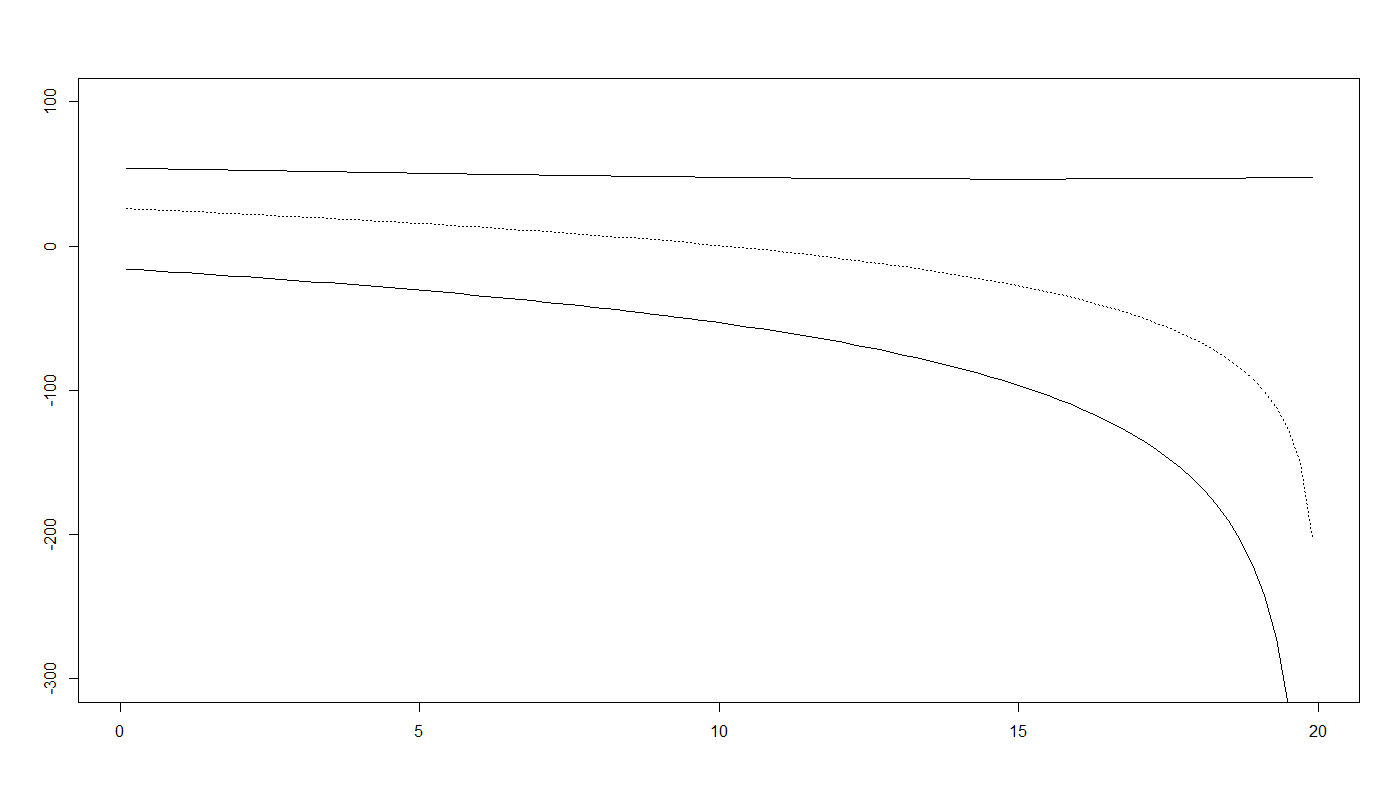}
\caption{$ \lambda \mapsto \bar x_1, x_1^*, \bar x_2$ for Scenario B \label{cs1lambda}}
\end{subfigure}
\begin{subfigure}[H]{0.49\textwidth}
\includegraphics[width=\textwidth]{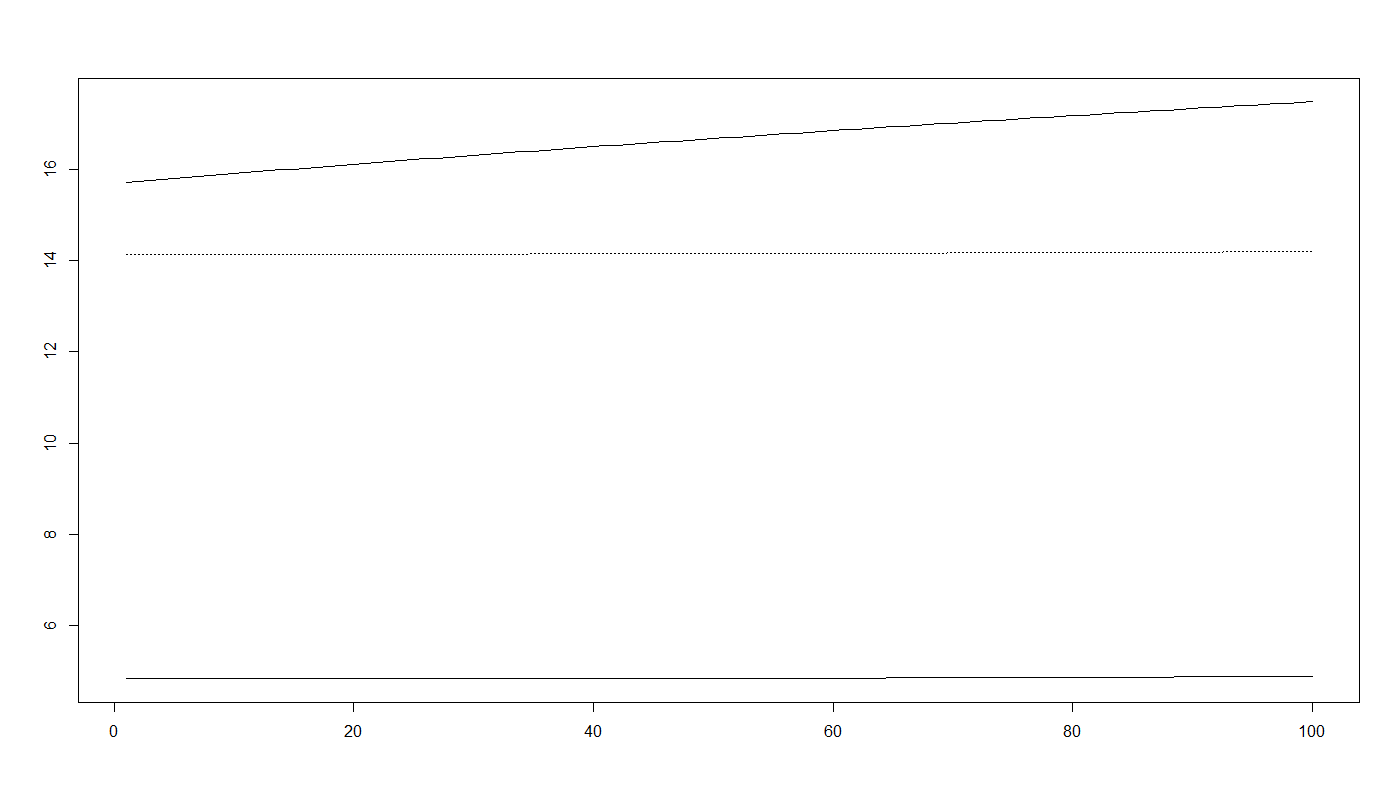}
\caption{$ d \mapsto \bar x_1, x_1^*, \bar x_2$ for Scenario B \label{cs1d} }
\end{subfigure}
\begin{subfigure}[H]{0.49\textwidth}
\includegraphics[width=\textwidth]{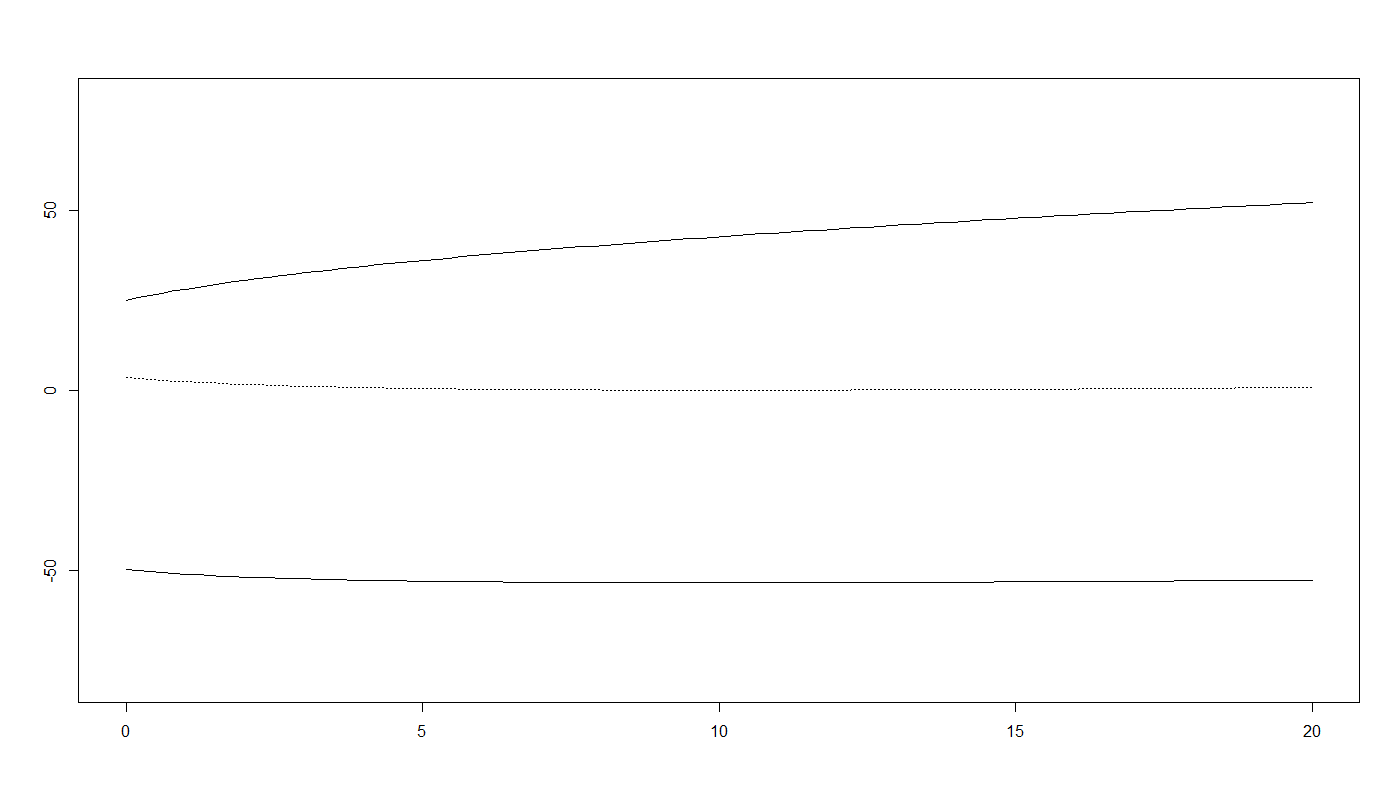}
\caption{$ \gamma \mapsto \bar x_1, x_1^*, \bar x_2$ for Scenario B \label{cs1gamma}}
\end{subfigure}
\end{figure}
Figures \ref{ne11graph}-\ref{ne12graph} show how the equilibrium payoff functions behave in the selected scenarios, with the dashed lines showing the smooth-pasting of the three components of the payoff in \eqref{system1} and \eqref{system2}. From Figure \ref{ne11graph} to Figure \ref{ne12graph} we can see how a reduction in the volatility seems to shrink the continuation region, hence, the players become more cautious, reducing their intervention regions when there is more uncertainty. Another interesting fact to note is how the relative distance between $\bar x_1$ and $\bar x_2$ becomes smaller. This can be due to the increase in P2's terminal payoff sensitivity, $b$, and the increase in P1's exogenous cost, $s$. In one direction, P2 is losing more money when she decides to terminate the game, therefore she will not stop when the state process value is too high, hence she reduces her threshold $ \bar x_2$. In the other,  since P1 is facing higher exogenous costs, she pushes the target, $x_1^*$, as far as she can, making sure the state process is not going too low, rising the barrier $\bar x_1$.

Figures \ref{cs1c}-\ref{cs1lambda}-\ref{cs1d}-\ref{cs1gamma} represent some comparative statics of the thresholds $\bar x_1, \, x_1^*$ and $\bar x_2$ for Scenario B. Similar graphs hold for Scenario A as well, therefore they are omitted. First, in Figure \ref{cs1c} we can observe how an increase in P1's fixed cost expands the gap between $\bar x_1$ and $x_1^*$. The more P1 has to pay at any intervention time, the less often she will intervene, lowering the threshold, $\bar x_1$, and increasing the target, $x_1^*$. This allows P2, who does not like high values of $x$, to slightly lower her threshold, $\bar x_2$, so as to pay less when she will stop the game. In Figure \ref{cs1lambda} the behavior with respect to the proportional cost is quite different. P1 will reduce the interventions for higher $\lambda$,  with the distance between $\bar x_1$ and $x_1^*$ left nearly unchanged, while P2 keeps the barrier at a constant level $\bar x_2$. In particular, P1 tends to never intervening when the proportional cost reaches its maximum, set by the condition $1 - \lambda r >0$. This behavior shows how P1 is quite indifferent to changes in the proportional cost when this is not too big while she is really sensitive once it gets high. Finally, in Figures \ref{cs1d}-\ref{cs1gamma} we can see that, when P2's gains more each time P1 intervenes increases, P2 is happy playing for longer, heightening the threshold $\bar x_2$, since she is receiving more money.

\paragraph{Equilibrium where P1 induces P2 to stop.} To satisfy \eqref{NE21}-\eqref{NE22}, we want $\hat w$ to be neither too high nor too low, in particular, high $\lambda$ should help in \eqref{NE21} as high $\hat w$ in \eqref{NE22}. As before, via graphical analysis it is possible to show that $\hat w$, solution to \eqref{G(w) second case}, is decreasing in $a, b, s$ and increasing in $c, d, q, \lambda$ and $\gamma$. Therefore, the first instance of Nash equilibrium, Scenario B, has been selected to have high $\lambda$ and $\hat w$, choosing high values of $c, d, q$ and $\gamma$ and low values of $b$ and $s$, whereas for Scenario A we have looked for lower values of  $\lambda$ and adapted the others.
\begin{center} \begin{tabular}{ c | c | c | c | c | c |c | c | c | c | c || c | c } 
   & $r$ & $\sigma$ & $c$ & $d$ & $\lambda$ & $\gamma$ & $a$ & $b$ & $s$ & $q$ & $\bar x_1$ &  $\bar x_2$\\
 \hline Scenario A & 0.01 & 5 & 100 & 100 & 25 & 10 & 24 & 9 & 45 & 0 & 22.56 & 32.68 \\
  Scenario B& 0.01 & 1.5 & 150 & 125 & 80 & 25 & 70 & 15 & 10 & 15 & 14.27 & 25.72. \\
\end{tabular}\end{center}
where $\bar x_1$ and $\bar x_2$ are as in Proposition \ref{propequi2}.
\begin{figure}[H]
\centering
\captionsetup[subfigure]{labelformat=simple}

\caption{Type II Equilibria}
\begin{subfigure}[H]{0.49\textwidth}
\includegraphics[width=\textwidth]{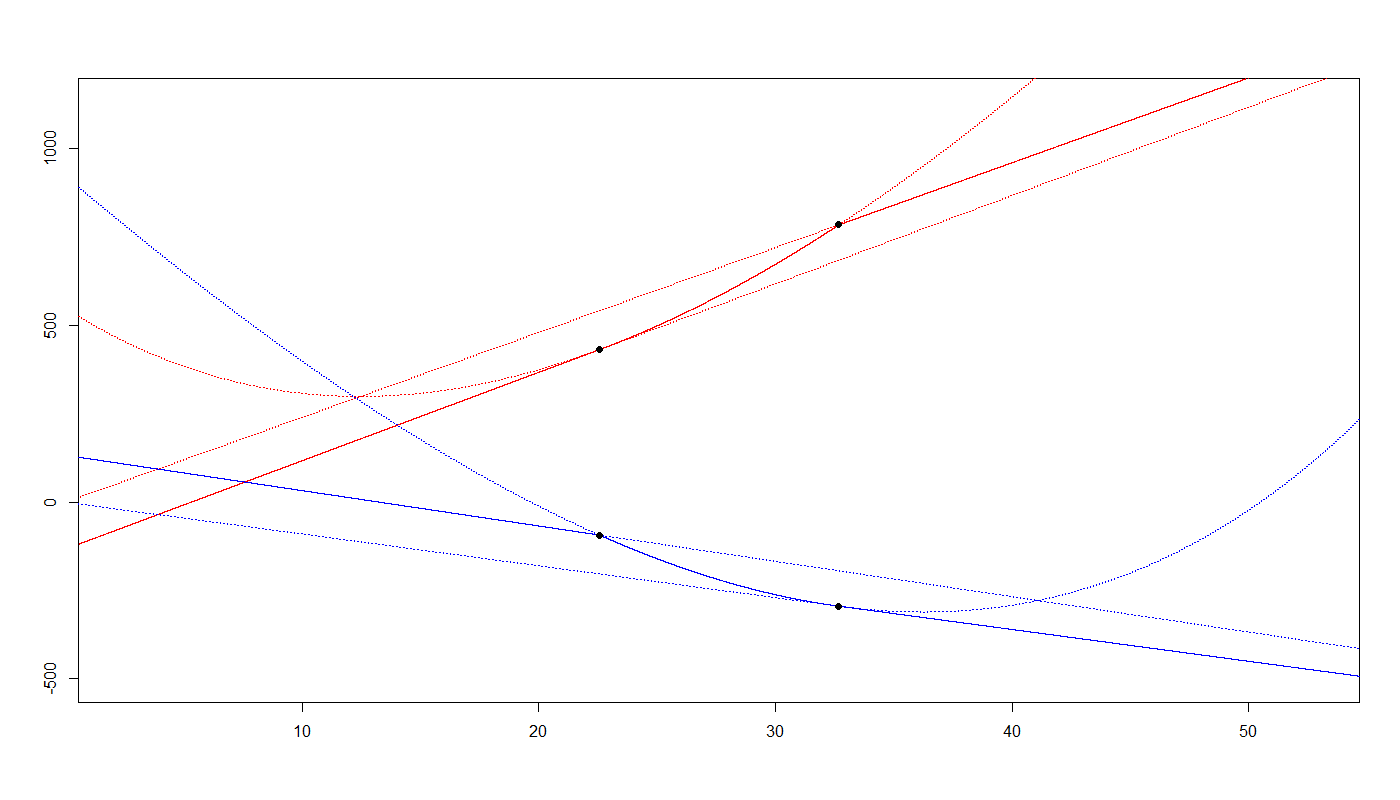}
\caption{$ x \!\mapsto\! V_1(x)$ in red, $ x \!\mapsto\! V_2(x)$ in blue for Scenario A \label{ne21graph}}
\end{subfigure}
\begin{subfigure}[H]{0.49\textwidth}
\includegraphics[width=\textwidth]{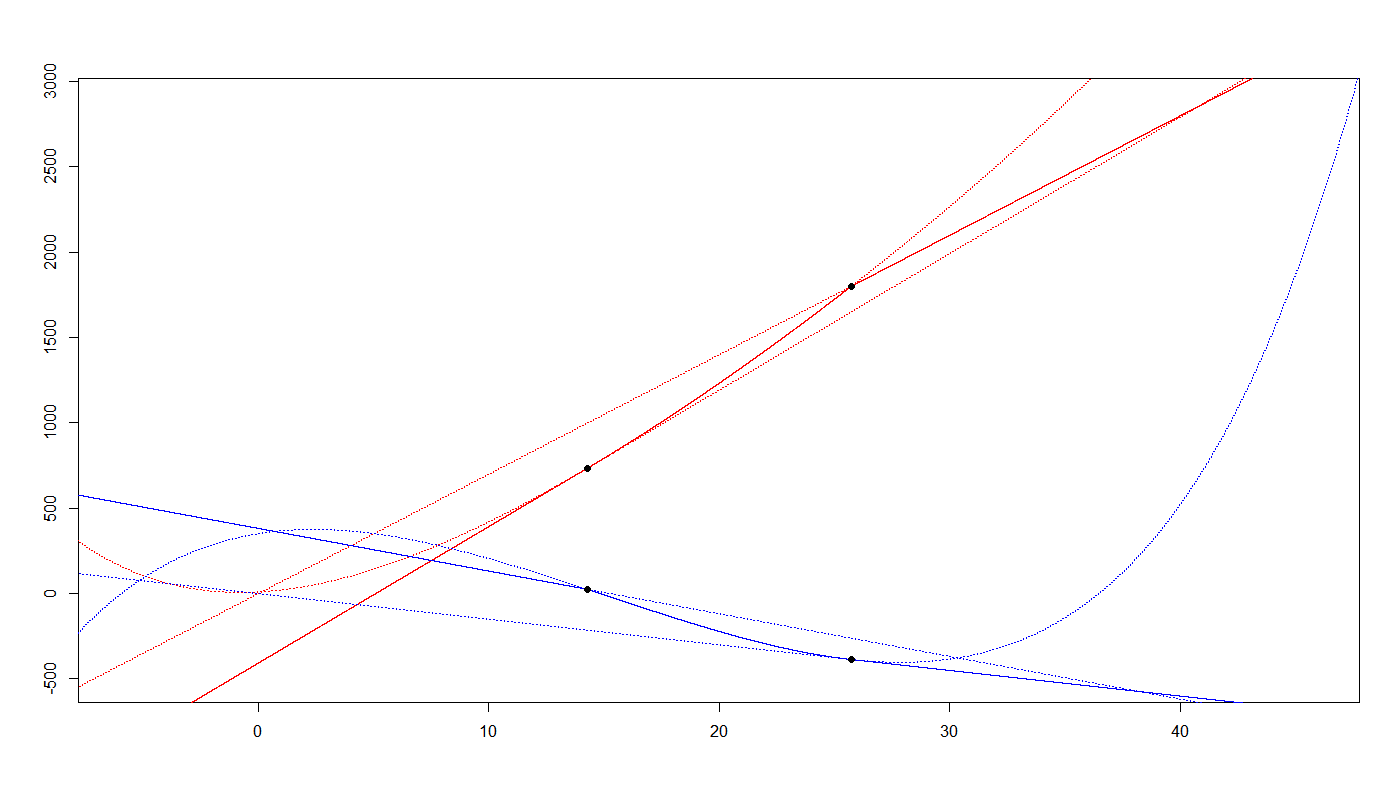}
\caption{$ x \!\mapsto\! V_1(x)$ in red, $ x \!\mapsto\! V_2(x)$ in blue for Scenario B \label{ne22graph}}
\end{subfigure}
\begin{subfigure}[H]{0.49\textwidth}
\includegraphics[width=\textwidth]{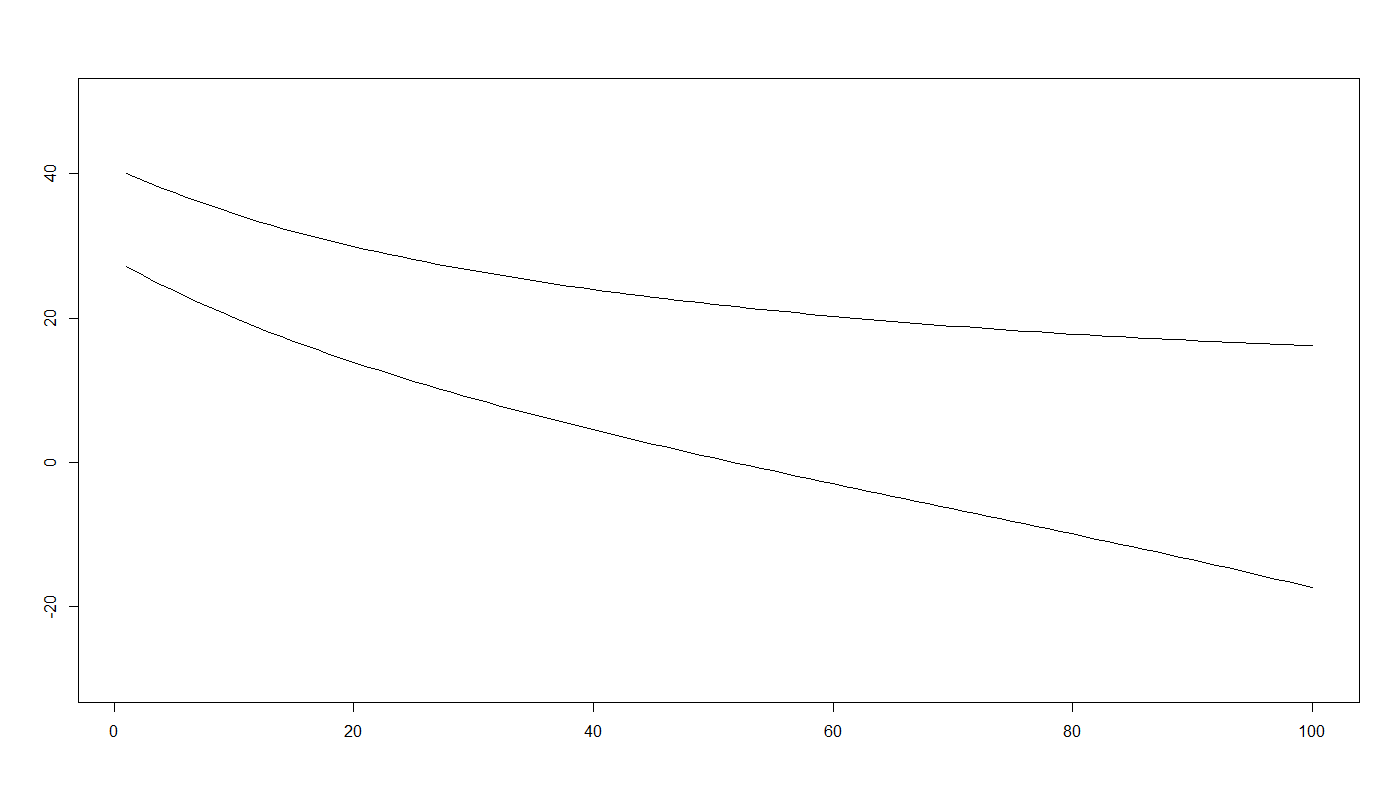}
\caption{$ c \mapsto \bar x_1, \bar x_2$ for Scenario B \label{cs2c}}
\end{subfigure}
\begin{subfigure}[H]{0.49\textwidth}
\includegraphics[width=\textwidth]{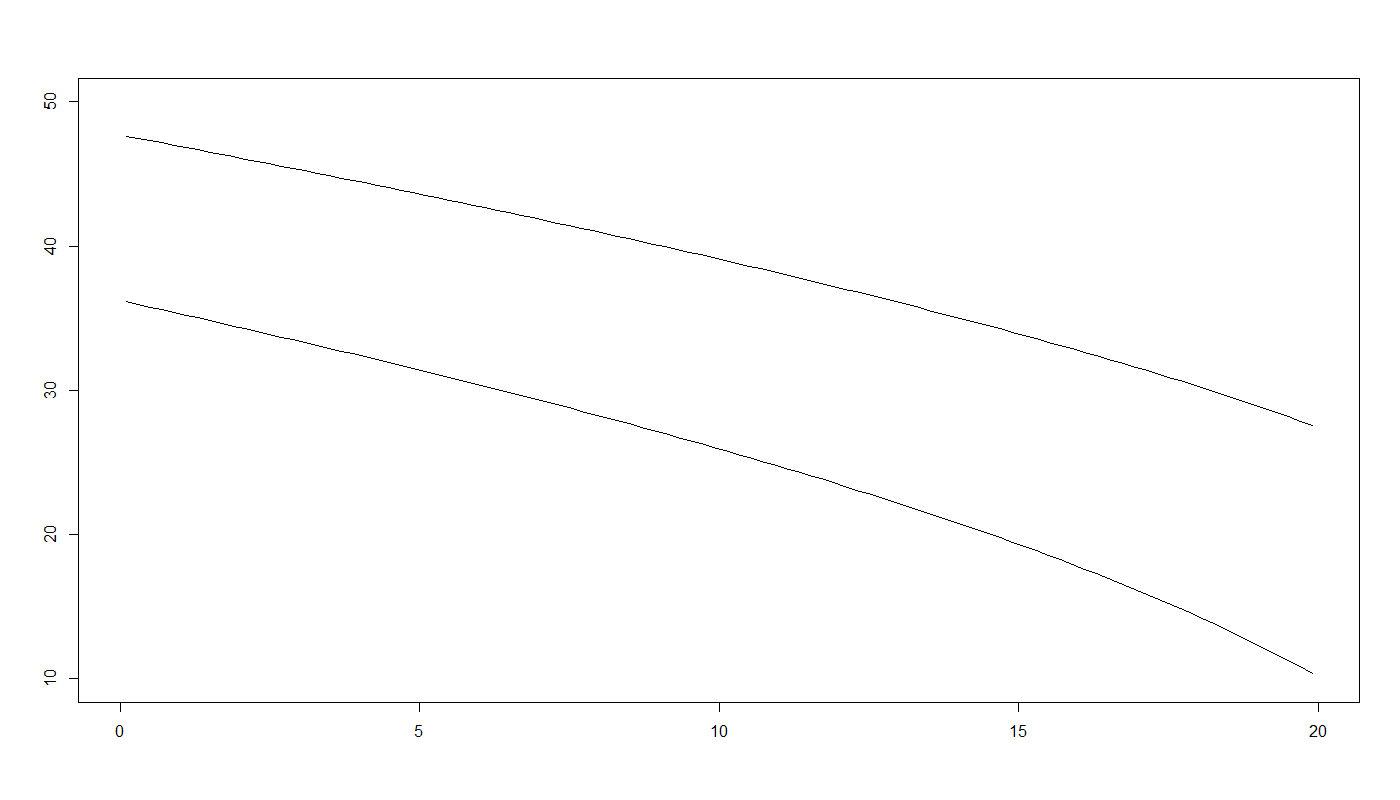}
\caption{$ \lambda \mapsto \bar x_1,  \bar x_2$ for Scenario B \label{cs2lambda} }
\end{subfigure}
\end{figure}
\clearpage
\begin{figure}\captionsetup[subfigure]{labelformat=simple}
\ContinuedFloat 
\begin{subfigure}[H]{0.49\textwidth}
\includegraphics[width=\textwidth]{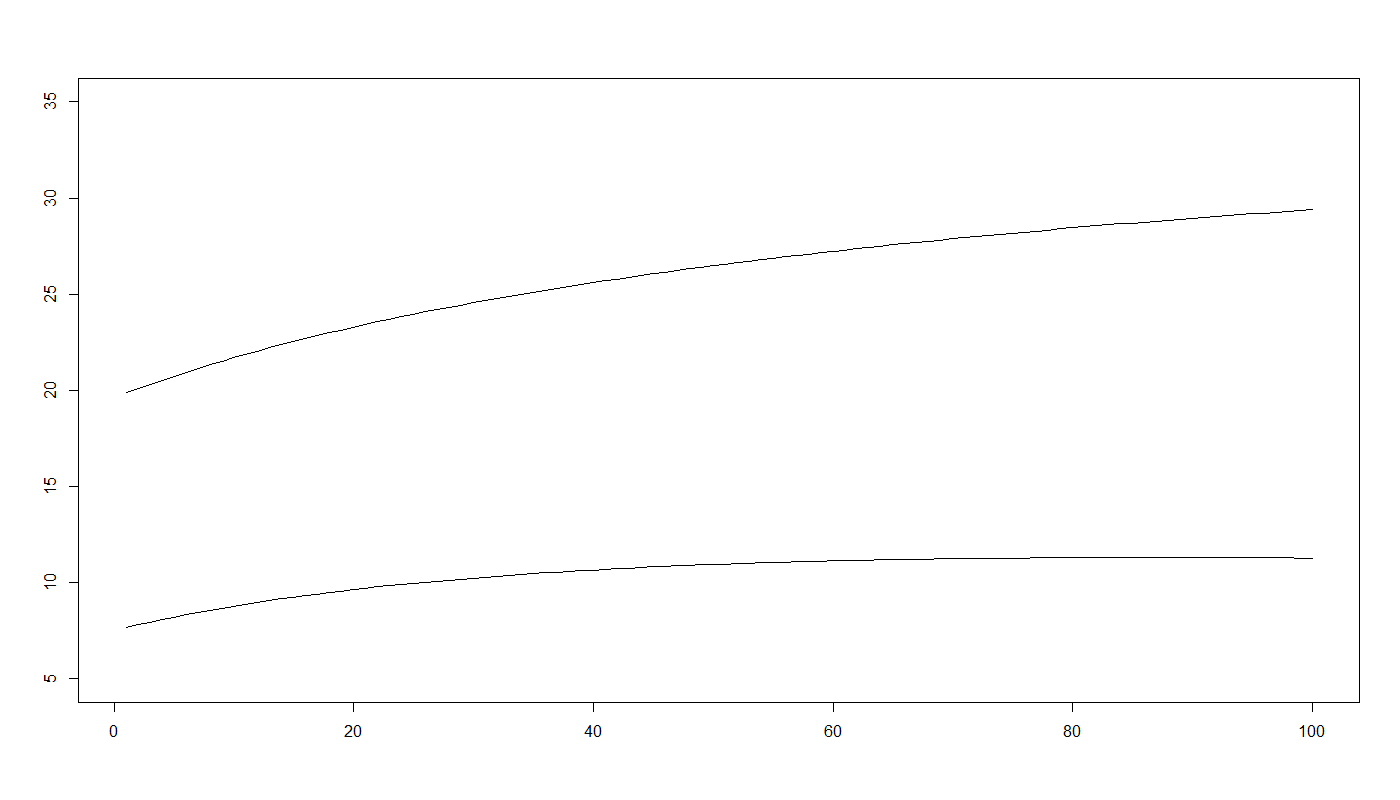}
\caption{$ d \mapsto \bar x_1,  \bar x_2$ for Scenario B \label{cs2d}}
\end{subfigure}
\begin{subfigure}[H]{0.49\textwidth}
\includegraphics[width=\textwidth]{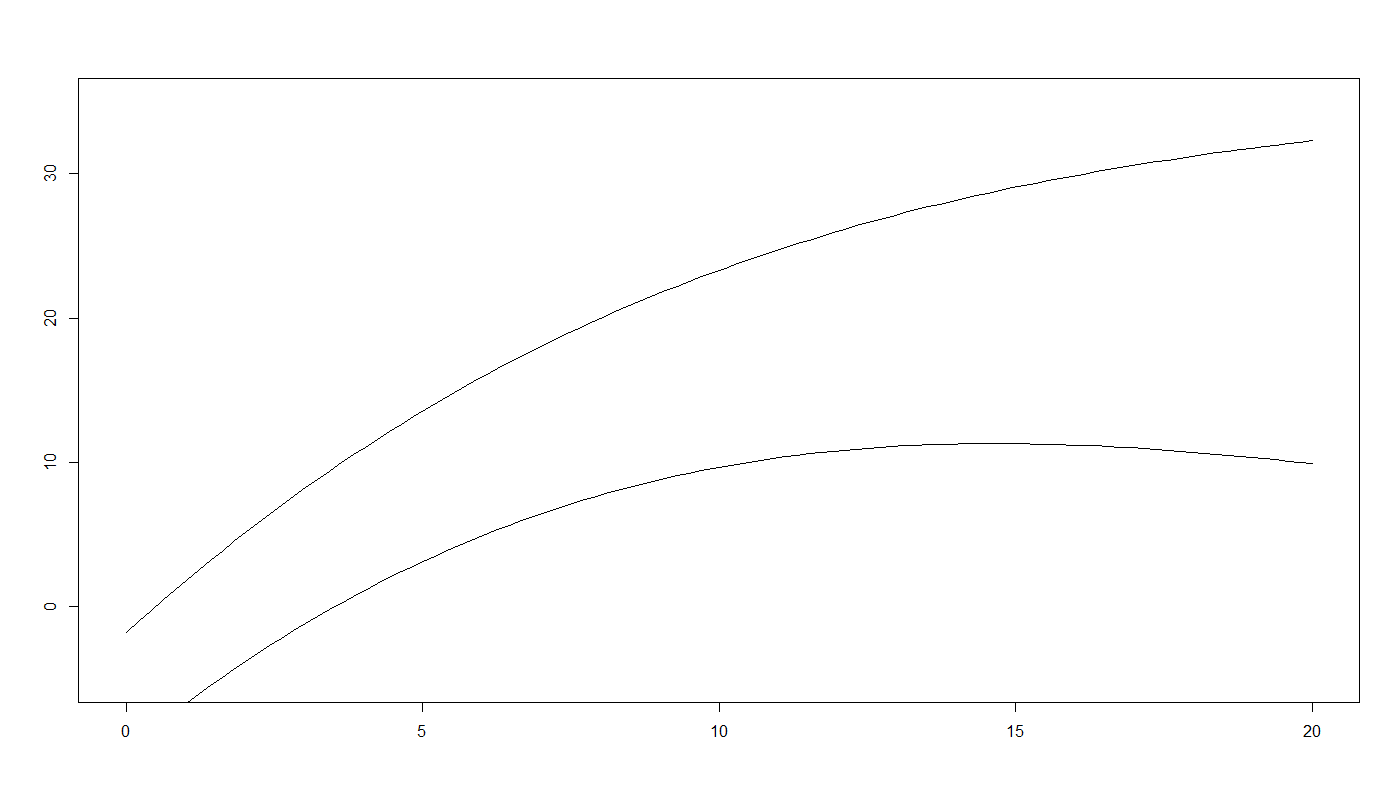}
\caption{$ \gamma \mapsto \bar x_1,  \bar x_2$ for Scenario B \label{cs2gamma}}
\end{subfigure}
\end{figure}
As before, Figure \ref{ne21graph}-\ref{ne22graph} represent the equilibrium payoff functions in the selected examples. First, we can observe that the continuation region in Scenario A is shifted to the right with respect to the one in Scenario B and we can observe that its width has not changed much from one case to the other. Furthermore, we can notice that Scenario B is more profitable for P2 and less profitable for P1. These two facts might be explained by the following changes from Scenario B to Scenario A:  P1's exogenous cost, $s$, increases, so P1 cannot tolerate low levels of $x$, increasing her threshold $\bar x_1$. Moreover, although P2's gains, $q$, $d$ and $\gamma$, decrease we do not see her threshold scale down as it would be expected as the game is now less profitable. This is probably due to $b$'s reduction, which leads P2 to stop for higher values of $\bar x_2$ since she is going to lose less when she decides to stop. 

Now, let us spend some words on the comparative statics in Figures \ref{cs2c}-\ref{cs2lambda}-\ref{cs2d}-\ref{cs2gamma}.  When P1's costs, $c$ and $\lambda$, increases, Figure \ref{cs2c}-\ref{cs2lambda}, P1 would intervene for lower values of $x$ and the distance $\bar x_2 - \bar x_1$ will increase, even though $\bar x_2$ gets lower as well. This can be explained as follows, with the costs increasing, P1 is less willing to intervene, reducing $\bar x_1$, even though this shift allows P2 to lower her threshold, $\bar x_2$, since she likes low values of $x$. When the fixed gain, $d$, rises, Figure \ref{cs2d}, P2 can afford the game to run for longer, increasing $\bar x_2$, as she will gain more when P1 will make her stop. Moreover, this makes P1 heighten $\bar x_1$ in order to limit the proportional costs increment. Lastly, we have a similar behavior to the one described above for the proportional gain, $\gamma$, Figure \ref{cs2gamma}. The main difference is the speed with which the distance between the thresholds increases, higher for proportional gain increments. This happens because, in case of proportional gain increments, P2 is more incentivized to push $\bar x_2$ far away since the bigger the impulse the more the revenue, whereas an analogous behavior in case of fixed gain increments would lead to a loss in the terminal payoff outrunning the additional profit due to the fact that the gain, $d$, does not depend on the intensity of the impulse P1 is playing while the losses are increasing, since they depend on P2's threshold, $-b\bar x_2$.

\paragraph{Comparison between the two equilibria} We conclude with a short discussion on the reasons why P1 would play aggressively, forcing P2 to stop. To do so we compare first the two scenarios A and B in both equilibria. So, going from Type I to Type II we see a reduction in the proportional gain, $\gamma$, an increase in P1 terminal payoff sensitivity, $b$, and a reduction in P2's exogenous gain, $q$, making P2 lower her threshold, $\bar x_2$, to reduce the losses at the end of the game. Then, P1's exogenous cost, $s$, increases making P1 rise both the threshold and the target, $\bar x_1$ and $x_1^*$ respectively. Furthermore, P1 terminal payoff sensitivity, $a$, increases and, intuitively incentivize P1 to let P2 end the game sooner so to receive the terminal payoff. More specifically, since $\tilde w$ is decreasing in $a$, its increase makes $\ln \tilde w = \theta(\bar x_2 - \bar x_1)$ decrease, hence, since the distance between the two thresholds is now smaller, P1's target, $x_1^*$, is closer to P2's barrier up to the point they coincide, $x_1^*\equiv\bar x_2$. 

Regarding Scenario B, again from Type I to Type II, we observe increments in the terminal payoff sensitivity of the two players, $a$ and $b$, in particular P1's sensitivity rises much more than in the first scenario, hence, P1 is more incentivized to let P2 end the game. Another important change regards the proportional cost, $\lambda$, which is very high in case P1 induces P2 to stop. As we have seen before in the comparative statics in Figure \ref{cs1lambda}, P1 intervenes less and less when the proportional cost becomes higher and higher, so it is more convenient to intervene only once, inducing P2 to stop.

We finally observe that while we have managed to find numerical values for which only one of the two types of Nash equilibria emerges at a time, the problem of whether the two equilibria can coexist remains open.

\bibliography{mybib}
\bibliographystyle{apa}
\appendix

\section{Auxiliary results} 
In this appendix we have gathered some technical results used in the verification parts of Section \ref{game} for both types of Nash equilibrium. We start with two lemmas on the continuations regions in the equilibrium where simultaneous actions are not allowed. 
\begin{lemma} \label{lemma w1 qvi} Let $W_1$ be as in \eqref{system1}. Then we have 
$$
\delta (x) = (x_1^* - x)\mathbf 1_{( -\infty, \, x_1^*]}(x), \quad x \in \mathbb{R}.
$$
Moreover
\begin{eqnarray}\label{cont-P1}
\{\mathcal{M} W_1 - W_1 < 0\} = (\bar x_1, \, \infty)\quad \text{and}\quad  \{\mathcal{M} W_1 - W_1 = 0\} = (- \infty, \, \bar x_1].
\end{eqnarray}
\end{lemma}

\begin{proof} By a simple change of variable we obtain
\begin{eqnarray*}
\mathcal{M}W_1 = \max_{\delta \geq 0} \{ W_1( x+ \delta) - c - \lambda \delta \} = \max_{y  \geq x} \{ W_1( y) - c - \lambda ( y - x ) \}.
\end{eqnarray*}
Let $\Gamma(y) := W_1(y) - \lambda y$. By definition of $W_1$ we have $\Gamma'(\bar x_1) = \Gamma'( x_1^*) = 0 $. Moreover, the following properties are satisfied:
\begin{enumerate}
\item $\Gamma'(x) = 0$ in $(-\infty, \, \bar x_1]$;
\item $\Gamma'(x) = a - \lambda < 0$ in $[\bar x_2, \, \infty)$;
\item $\Gamma'(x) > 0$ (resp. $<0$) in $( \bar x_1, \, x_1^*)$ (resp. in $( x_1^*, \, \bar x_2)$). \end{enumerate}
Properties (i) and (ii) are easily checked. Regarding (iii), recall that
\[ \Gamma'(x) = \varphi_1'(x) - \lambda = \theta C_{11} e^{\theta x} - \theta C_{12} e^{- \theta x} + \frac{1}{r} - \lambda , \quad x \in (\bar x_1, \bar x_2). \] 
To study its sign, notice that $\Gamma''(x) = \theta^2 C_{11} e^{\theta x} + \theta^2 C_{12} e^{- \theta x} >0$ for all  $x  \in  (\bar x_1,    \tilde x)$,
where $\tilde x$ is such that $e^{\theta \tilde x} = \sqrt{ -C_{12}/C_{11}} = e^{\frac{\theta}{2}(x_1^* + \bar x_1)}$. Moreover, since $\tilde x < x_1^*$ we have $\Gamma''(x_1^*)<0$. Hence, it follows that $\Gamma'(x) > 0$ in $( \bar x_1, \, x_1^*)$, while $\Gamma'(x) < 0$ in $( x_1^*, \, \bar x_2)$.

As a consequence, $\Gamma$ has a unique global maximum in $ x_1^*$, so that
\[
\max_{y \geq x} \Gamma(y) = \left\{ \begin{array}{ll} \Gamma(x_1^*) & \mbox{ in } \; ( -\infty, \, x_1^*] \\ \Gamma(x) & \mbox{ in} \; (x_1^*, \, \infty) \end{array}\right.
\]
which gives
$$
\argmax_{\delta \geq 0} \{ W_1( x+ \delta) - c - \lambda \delta \} = \left\{ \begin{array}{ll} \{ x_1^* - x \} & \mbox{ in } \; ( -\infty, \, x_1^*] \\ \{ 0 \} & \mbox{ in } \; (x_1^*, \, \infty) \end{array} \right. 
$$
This implies the first part of our statement, i.e. $\delta(x) = (x_1^* - x)\mathbf 1_{( -\infty, x_1^*]}(x)$. Now, to show \eqref{cont-P1}, notice first that
\begin{eqnarray*}
&&\mathcal{M} W_1(x) = \left\{ \begin{array}{ll} W_1(x) - \zeta(x) & \mbox{ in } \; ( \bar x_1, \infty) \\ W_1(x) & \mbox{ in } \; ( - \infty, \, \bar x_1], \end{array} \right. 
\end{eqnarray*} 
where we set
$$
\zeta(x) := \left\{ \begin{array}{ll} \varphi_1(x) - \varphi_1(x_1^*) + c + \lambda( x_1^* - x) & \mbox{ in } \; (\bar x_1, \, x_1^*] \\ c & \mbox{ in } \: ( x_1^*, \, \infty) \end{array} \right.
$$
Now, we are going to prove that $\zeta >0$. By $C^0$-pasting in $\bar x_1$ we have $ \varphi_1( \bar x_1) = \varphi( x_1^*) - c - \lambda (x_1^* - \bar x_1)$, therefore 
$$ 
\zeta(x)= \varphi_1(x) - \varphi_1(\bar x_1) - \lambda (\bar x_1 - x) = \Gamma(x) - \Gamma(\bar x_1), \quad x \in (\bar x_1, x_1^*],
$$ 
which is strictly positive since $\Gamma$ is increasing in $(\bar x_1, \, x_1^*]$. Hence, $\zeta$ is strictly positive and we have 
\begin{eqnarray*}
\{\mathcal{M} W_1 - W_1 < 0\} = (\bar x_1, \, \infty), \quad \{\mathcal{M} W_1 - W_1 = 0\} = (- \infty, \, \bar x_1].
\end{eqnarray*}
\qed
\end{proof}   

\begin{lemma} \label{lemma2} Let $W_2$ be as in \eqref{system2}. Assume there exists a solution $(\tilde z, \, \tilde w)$ to \eqref{F(z)}-\eqref{wequation} such that $1 < \tilde z < \tilde w$ and
\begin{eqnarray*}
&& 0 \le  \frac{(1 - br)(1 - \lambda r)(\tilde w^2 - \tilde z)}{\theta \tilde w(1 - ar)(\tilde z + 1)} + \frac{1 - br}{1 - ar}s  - q < \frac{1 - br}{\theta},\\
&& \left(\frac{1 - br}{1 - ar}\left( \frac{(1 - \lambda r)(\tilde w^2 - \tilde z)}{\theta \tilde w(\tilde z + 1)} + s\right) - q\right)(\tilde w - 1)^2 + \frac{1 - br}{\theta}(1 + 2\tilde w \ln \tilde w - \tilde w^2) > 0.
\end{eqnarray*} 
Then, we have 
\begin{eqnarray*}
&& \{x \in \mathbb R :-bx - W_2(x) < 0\} = (- \infty, \, \bar x_2),\\
&& \{x \in \mathbb R :-bx - W_2(x) = 0\} = [\bar x_2, \, \infty).
\end{eqnarray*}
\end{lemma}
\begin{proof}
First, we recall that
$$
W_2(x) = \left\{ \begin{array}{ll} - bx & \mbox{ in } \; [ \bar x_2, \, \infty) \\ \varphi_2(x) & \mbox{ in } (\bar x_1, \, \bar x_2) \\ \mathcal{H} W_2(x) & \mbox{ in } (-\infty, \, \bar x_1] \end{array} \right.
$$
where 
$$
\varphi_2(x) = C_{21} e^{\theta x} + C_{22} e^{-\theta x} + \frac{q - x}{r}.
$$
We want to prove that $\varphi_2(x) > - bx $ in $(\bar x_1, \, \bar x_2)$ and $ \mathcal{H} W_2(x) > -bx $ in $(- \infty, \, \bar x_1]$.
For the first inequality we are interested in the conditions such that, for all $x \in (\bar x_1, \, \bar x_2)$, we have 
\begin{equation}\label{phi2positive}
C_{21} e^{\theta x} + C_{22}e^{-\theta x} + \frac{q - (1 - br)x}{r} > 0 ,
\end{equation} 
or, equivalently,
$$
 e^{\theta (x - \bar x_2)} \left[(1 - br)\left( \frac{1}{\theta} + \bar x_2\right) - q \right] + e^{\theta (\bar x_2 - x)} \left[(1 - br)\left( \bar x_2 - \frac{1}{\theta}\right)  - q\right] + 2\left( q - ( 1 - br)x\right) > 0.
$$
Now, applying the change of variable $e^{\theta(\bar x_2 - x)} = z > 1$ to the inequality above yields 
$$(1 - br)\left(\bar x_2 + \frac{1}{\theta}\right) - q  + z^2\left[(1 - br)\left(\bar x_2 - \frac{1}{\theta}\right) - q \right] + 2z\left(q - (1 - br)\bar x_2 + \frac{1 - br}{\theta}\ln z\right) > 0.$$
Since $\ln z >0$ and $1-br>0$ by assumption, the left-side above is bigger than
\[ (1 - br)\left(\bar x_2 + \frac{1}{\theta}\right) - q  + z^2\left[(1 - br)\left(\bar x_2 - \frac{1}{\theta}\right) - q \right] + 2z\left(q - (1 - br)\bar x_2 \right),\]
which is quadratic in $z$ and it can be factorized as  
$$ (z - 1)\left(z - \frac{(1 - br)\left(\bar x_2 + \frac{1}{\theta}\right) - q}{(1 - br)\left(\bar x_2 - \frac{1}{\theta}\right) - q}\right) .$$
We are going to show that our assumptions grant that the expression above is positive, which in turn will imply \eqref{phi2positive}.
Hence, the second factor is positive if the following holds:
$$(1 - br)\left(\bar x_2 - \frac{1}{\theta}\right) - q < 0, \quad (1 - br) \bar x_2 - q \geq 0.$$
Then, using \eqref{barx2}, the two inequalities above can be rewritten as 
$$
0 \le \frac{(1 - br)(1 - \lambda r)(\tilde w^2 - \tilde z)}{\theta \tilde w(1 - ar)(\tilde z + 1)} + \frac{1 - br}{1 - ar}s  - q < \frac{1 - br}{\theta},
$$ which is true by assumption.

For showing the second inequality,  i.e. $ \mathcal{H} W_2(x) > -bx $ in $(- \infty, \, \bar x_1]$, we observe first that
\begin{equation}\label{2nd-ineq}
\varphi_2(x_1^*) + d + \gamma (x_1^* - x) > - bx ,\quad x \, \in ( -\infty, \, \bar x_1].
\end{equation}
From the $C^0$-pasting condition in $\bar x_1$ we have that $\varphi_2(\bar x_1) = \varphi_2(x_1^*) + d + \gamma ( x_1^* - \bar x_1)$, therefore we can rewrite \eqref{2nd-ineq} as
$$
\varphi_2(\bar x_1) + \gamma (\bar x_1 - x) > - bx .
$$
Since $ b < \gamma $ we only need to check that $F(\bar x_1)>0$:
\begin{align*}
F(\bar x_1) & = \varphi_2(\bar x_1) + b \bar x_1 = C_{21} e^{\theta \bar x_1} + C_{22}e^{- \theta \bar x_1} + \frac{ q - (1 - br) \bar x_1}{r} ,\\
& = e^{- \theta( \bar x_2 - \bar x_1)} \left[(1 - br)\left(\bar x_2 + \frac{1}{\theta}\right) - q\right] + e^{\theta(\bar x_2 - \bar x_1)}\left[ (1 - br)\left(\bar x_2 - \frac{1}{\theta}\right) - q\right] + 2(q - ( 1 - br)\bar x_1) .
\end{align*}
Now, using again the change of variable $w=e^{\theta(\bar x_2 - \bar x_1)}$, we have $\bar x_1 = \bar x_2 - \frac{\ln w}{\theta}$ and so $F(\bar x_1)e^{\theta(\bar x_2 - \bar x_1)}$ can be re-expressed as
$$
((1 - br)\bar x_2 - q)(\tilde w - 1)^2 + \frac{1 - br}{\theta}(1 + 2\tilde w\ln \tilde w - \tilde w^2),
$$
 which, using \eqref{barx2}, can be rewritten as  
$$ \left(\frac{1 - br}{1 - ar}\left( \frac{(1 - \lambda r)(\tilde w^2 - \tilde z)}{\theta \tilde w(\tilde z + 1)} + s\right) - q\right)(\tilde w - 1)^2 + \frac{1 - br}{\theta}(1 + 2\tilde w \ln \tilde w - \tilde w^2) ,$$
which is positive by assumption. \qed 
\end{proof}

We conclude the appendix with two more lemmas on similar results for the other kind of equilibrium, where P1 forces P2 to stop the game.

\begin{lemma} \label{lemma2 w1 qvi} Let $W_1$ be as in \eqref{system1}. Assume there exists a solution $\hat w $ to \eqref{G(w) second case} such that
$$ (1 - \lambda r)(w - w\ln w - 1) + cr\theta w > 0 . $$
Then we have 
$$
\delta (x) = (\bar x_2 - x)\mathbf 1_{( -\infty, \, \bar x_2]}(x), \quad x \in \mathbb{R}.$$
Moreover, we have
\begin{eqnarray*}
\{\mathcal{M} W_1 - W_1 < 0\} = (\bar x_1, \, \infty), \quad \{\mathcal{M} W_1 - W_1 = 0\} = (- \infty, \, \bar x_1].
\end{eqnarray*}
\end{lemma}

\begin{proof} First, observe that
\begin{eqnarray*}
\mathcal{M}W_1 = \max_{\delta \geq 0} \{ W_1( x+ \delta) - c - \lambda \delta \} = \max_{y  \geq x} \{ W_1( y) - c - \lambda ( y - x ) \}.
\end{eqnarray*}
Let us denote $\Gamma(y) := W_1(y) - \lambda y$. By definition of $W_1$ we have $\Gamma'(\bar x_1)  = 0 $. Moreover, the following properties hold true:
\begin{enumerate}
\item $\Gamma'(x) = 0$ in $(-\infty, \, \bar x_1]$;
\item $\Gamma'(x) = a - \lambda < 0$ in $[\bar x_2, \, \infty)$;
\item $\Gamma'(x)>0$ in $(\bar x_1, \bar x_2)$.
\end{enumerate}
As properties (i) and (ii) can be easily checked, we turn to showing (iii). Observe that, for all $x \in (\bar x_1, \bar x_2)$, one has 
$\Gamma'(x) = \varphi_1'(x) - \lambda = \theta C_{11} e^{\theta x} - \theta C_{12} e^{- \theta x} + \frac{1}{r} - \lambda $, hence
\begin{align*}
\Gamma'(x) =& \frac{\theta}{2} e^{\theta(x - \bar x_1)} \left[ (a - \lambda) \bar x_2 - \left(\bar x_1 + \frac{1}{\theta}\right)\frac{1 - \lambda r}{r} - c  \frac{s}{r}\right] \\
& -  \frac{\theta}{2} e^{-\theta(x - \bar x_1)} \left[ (a - \lambda) \bar x_2 - \left(\bar x_1 - \frac{1}{\theta}\right)\frac{1 - \lambda r}{r} - c + \frac{s}{r}\right] + \frac{1 - \lambda r}{r}. 
\end{align*}
Using the fact that $\bar x_1 = \bar x_2 - \frac{\ln \hat w}{\theta}$ and setting $z=e^{\theta(x - \bar x_1)}$ we can rewrite it as
$$
\left( - ( 1 - ar) \bar x_2 + \frac{1 - \lambda r}{\theta}\ln \hat w - cr + s\right)(z^2 - 1) - \frac{1 - \lambda r}{\theta}(z - 1)^2 > 0,
$$
which can be factorized as 
$$ ( z - 1)\left(z + \frac{\frac{1 - \lambda r}{\theta}(\ln \hat w + 1) - (1 - ar)\bar x_2 - cr + s}{\frac{1 - \lambda r}{\theta}(\ln \hat w - 1) - (1 - ar)\bar x_2 - cr + s}\right)>0 ,$$
which is true whenever $\frac{1 - \lambda r}{\theta}(\ln \hat w - 1) - (1 - ar)\bar x_2 - cr + s > 0$. Therefore, recalling \eqref{x2 bar 1}, after some algebraic manipulation, we obtain the equivalent condition 
$$(1 - \lambda r)(\hat w - \hat w\ln \hat w - 1) + cr\theta \hat w > 0 .$$
Hence property (iii) is fulfilled.  
As a consequence, $\Gamma$ has a unique global maximum point in $\bar  x_2$, and the rest of the proof follows the same lines as for Lemma \ref{lemma w1 qvi}. Hence, the details are omitted.
\qed
\end{proof}   

\begin{lemma} \label{lemma22} Let $W_2$ be as in \eqref{system2}. For every $ x \in \mathbb{R} $, assume there exists a solution $\hat w$ to \eqref{G(w) second case} such that:
$$
0 \le (1 - br)(\hat w^2 - 1) + 2(\theta r d - ( 1 - \gamma r) \ln \hat w) \hat w < (1 - br)(\hat w - 1)^2 .
$$
Then, we have 
\begin{eqnarray*}
\{x \in \mathbb R : W_2(x) > -bx\} = (- \infty, \, \bar x_2),\quad \{x \in \mathbb R : W_2(x) = -bx\} = [\bar x_2, \, \infty).
\end{eqnarray*}
\end{lemma}
\begin{proof}
First, recall that
$$
W_2(x) = \left\{ \begin{array}{ll} - bx & \mbox{ in } \; [ \bar x_2, \, \infty) \\ \varphi_2(x) & \mbox{ in } (\bar x_1, \, \bar x_2) \\ \mathcal{H} W_2(x) & \mbox{ in } (-\infty, \, \bar x_1] \end{array} \right.
$$
where 
$$
\varphi_2(x) = C_{21} e^{\theta x} + C_{22} e^{-\theta x} + \frac{q - x}{r}.
$$
Hence, we want to prove that $\varphi_2(x) > - bx $ in $(\bar x_1, \, \bar x_2)$ and $ \mathcal{H} W_2(x) > -bx $ in $(- \infty, \, \bar x_1]$.
For the first inequality we are interested in the conditions granting
$$
C_{21} e^{\theta x} + C_{22}e^{-\theta x} + \frac{q - (1 - br)x}{r} > 0 ,\quad x \in (\bar x_1, \, \bar x_2),
$$ 
or equivalently
$$
e^{-\theta(\bar x_2 - x)} \left[ (1 - br)\left( \bar x_2 + \frac{1}{\theta}\right) - q\right] + e^{\theta(\bar x_2 - x)}\left[(1 - br)\left( \bar x_2 - \frac{1}{\theta}\right) - q\right] + 2(q - (1 - br)x) > 0.
$$
Letting $z=e^{\theta(\bar x_2 - x)}$, the above inequality holds whenever
$$
 (1 - br)\left( \bar x_2 + \frac{1}{\theta}\right) - q + \left[(1 - br)\left( \bar x_2 - \frac{1}{\theta}\right) - q\right]z^2 + 2\left (q - (1 - br)\left(\bar x_2 - \frac{\ln z}{\theta}\right)\right)z > 0,
$$
Since $\ln z >0$ and $1-br>0$ by assumption, the left-side above is bigger than
\[ (1 - br)\left(\bar x_2 + \frac{1}{\theta}\right) - q  + z^2\left[(1 - br)\left(\bar x_2 - \frac{1}{\theta}\right) - q \right] + 2z\left(q - (1 - br)\bar x_2 \right),\]
which can be factorized as in the proof of Lemma \ref{lemma2} in
$$ (z - 1)\left(z - \frac{(1 - br)\left(\bar x_2 + \frac{1}{\theta}\right) - q}{(1 - br)\left(\bar x_2 - \frac{1}{\theta}\right) - q}\right).$$
We are going to show that our assumptions grant that the expression above is positive. We proceed as in the proof of Lemma \ref{lemma2}: the second factor above is positive if the following holds 
$$(1 - br)\left(\bar x_2 - \frac{1}{\theta}\right) - q < 0, \quad (1 - br)\bar x_2  - q \geq 0,$$
which, using \eqref{x2 bar 2}, can be rewritten as 
$$
0 \le (1 - br)(\hat w^2 - 1) + 2(\theta r d - ( 1 - \gamma r) \ln \hat w) \hat w < (1 - br)(\hat w - 1)^2 ,
$$ which is true by assumption.

For the second inequality we have
$$
- b \bar x_2 + d + \gamma (\bar x_2 - x) > - bx, \quad x \in ( -\infty, \, \bar x_1].
$$
Since $\gamma > b$, the inequality holds whenever $
(\gamma - b) (\bar x_2 - \bar x_1) + d > 0$,
which is always true since $\bar x_2 > \bar x_1$ by the ordering condition. \qed
\end{proof}

\end{document}